\documentclass[a4paper, leqno]{article}
\usepackage[english]{babel}
\usepackage{amsmath,amscd,amssymb,amsthm}
\usepackage{pb-diagram}
\usepackage{graphicx}
\usepackage{bm}
\usepackage{tikz} 
\usepackage{centernot,cancel}
\newtheorem{definition}{Definition}[section]
\newtheorem{theorem}{Theorem}[section]
\newtheorem{lemma}{Lemma}[section]
\newtheorem{proposition}{Proposition}[section]
\newtheorem{corollary}{Corollary}[section]

\newtheorem{remark}{Remark}[section]
\begin{document}
\title{Harmonic Bundles with Symplectic Structures}
\author{Takashi Ono\thanks{Department of Mathematics, Graduate School of Science, Osaka University, Osaka, Japan, u708091f@ecs.osaka-u.ac.jp}}

\date{}
\maketitle
\begin{abstract}
We study harmonic bundles with an additional structure called symplectic structure. We study them for the case of the base manifold is compact and non-compact. For the compact case, we show that a harmonic bundle with a symplectic structure is equivalent to principle $\mathrm{Sp}(2n,\mathbb{C})$-bundle with a reductive flat connection. For the non-compact case, we show that a polystable good filtered Higgs bundle with a perfect skew-symmetric pairing is equivalent to a good wild harmonic bundle with a symplectic structure.
\end{abstract}

\noindent
MSC: 53C07, 14J60\\
Keywords: harmonic bundles, non-abelian Hodge correspondence, Kobayashi-Hitchin correspondence
\section{Introduction}
\subsection{Harmonic bundles}
Let $X$ be a complex manifold. A Higgs bundle $(E,\overline\partial_E, \theta)$ over $X$ is a pair such that  $\mathrm{(i)}$ $(E,\overline\partial_E)$ is a holomorphic bundle on $X$ and $(\mathrm{ii})$ $\theta$ is a holomorphic 1-form which takes value in $\mathrm{End}(E)$ and satisfies the integrability condition $\theta\wedge\theta$=0. $\theta$ is called the Higgs field. It was introduced by Hitchin \cite{H} and Simpson generalized it to higher dimensional case \cite{S1}. Although there are many interesting objects that are related to the Higgs bundle, in this paper, we focus on $\textit{harmonic bundles}$.\par
Let $h$ be a hermitian metric of $E$. Let $\partial_h$ be the (1,0)-part of the Chern connection with respect to $\overline\partial_E$ and $h$. Let $\theta_h^\dagger$ be the adjoint of $\theta$ with respect to $h$. We obtain a connection $\nabla_h:=\partial_h+\overline\partial_E+\theta+\theta_h^\dagger$. We call a metric $h$ a $\textit{pluri-harmonic metric}$ if $\nabla_h^2=0.$ We call a bundle $(E,\overline\partial_E,\theta, h)$ a harmonic bundle if $h$ is a harmonic metric.  When $X$ is compact and K\"ahler, Hitchin and Simpson proved that the existence of such metrics is equivalent to the stability condition of Higgs bundles. This correspondence is called \textit{Kobayashi-Hitchin correspondence}.
\begin{theorem} [\cite{H,S1}\label{HS}]
Suppose $X$ is a compact K\"ahler manifold. $(E,\overline\partial_E,\theta)$ admits a harmonic metric if and only $(E,\overline\partial_E, \theta)$ is a polystable Higgs bundle and $c_1(E)=c_2(E)=0$. If $h_1$ and $h_2$ are harmonic metrics, then there exists a decomposition $(E,\overline\partial_E, \theta)=\oplus_i(E_i,\overline\partial_{E_i}, \theta_i)$ such that (i) the decomposition is orthonormal with respect to both $h_1$ and $h_2$ $(ii)$ there exist an $a_i>0$ such that  $h_1|_{E_i}=a_ih_2|_{E_i}$ for each $i$.
\end{theorem} 
We next recall the Kobayashi-Hitchin correspondence for the non-compact case. Let $X$ be a smooth projective variety over $\mathbb{C}$, $H$ be a normal crossing divisor of $X$, and $L$ be an ample line bundle of $X$. Kobayashi-Hitchin correspondence for the non-compact case is the correspondence between good wild harmonic bundles $(E,\overline\partial_E,\theta,h)$ on $X-H$
and  $\mu_L$-polystable good filtered Higgs bundles $(\mathcal{P}_*\mathcal{V},\theta)$ with vanishing Chern classes on $(X,H)$. See \cite{M0, S1,S2} for details about filtered bundles. In Section \ref{f sh0}, we briefly recall them.\par
The study of harmonic bundles for the non-compact case was initiated in \cite{S1, S2}. Simpson studied them on curves and when the Higgs field has the singularity called $\textit{tame}$. He established the Kobayashi-Hitchin correspondence in this case. In \cite{BB}, Biquard-Boalch studied the harmonic bundles on curves when the Higgs field admits a singularity called $\textit{wild}$ and proved the correspondence. In \cite{M1,M2}, Mochizuki fully generalized the correspondence for the higher dimensional case.\par
\begin{theorem}[\cite{BB,M1, M2,S1,S2}]\label{SBBM}
Let $X$ be a smooth projective variety, $H$ be a normal crossing divisor of $X$, and $L$ be an ample line bundle of $X$. Let $(E,\overline\partial_E, \theta, h)$ be a good wild harmonic bundle on $X-H$. Then $(\mathcal{P}_*^hE,\theta)$ is a $\mu_L$-polystable good filtered Higgs bundle with $\mu_L(\mathcal{P}_*^hE)=0$ and $\int_X\mathrm{ch}_2(\mathcal{P}_*^hE)c_1(L)^{\mathrm{dim}X-2}=0$.\par
Conversely, let $(\mathcal{P}_*\mathcal{V},\theta)$ be a $\mu_L$-polystable good filtered Higgs bundle satisfying the following vanishing condition:
\begin{equation}\label{vanish0}
 \mu_L(\mathcal{P}_{\ast}\mathcal{V})=0, \int_X\mathrm{ch}_2(\mathcal{P}_{\ast}\mathcal{V})c_1(L)^{\mathrm{dim}X-2}=0.
  \end{equation}
   Let $(E,\overline\partial_E, \theta)$ be the Higgs bundle which we obtain from the restriction of $(\mathcal{P}_*\mathcal{V},\theta)$ to $X- H$. Then there exists a pluri-harmonic metric $h$ for $(E,\overline\partial_E,\theta)$ such that $(\mathcal{V},\theta)|_{X\backslash H}\simeq(E,\theta)$ extends to $(\mathcal{P}_{\ast}\mathcal{V},\theta)\simeq(\mathcal{P}^h_{\ast}E,\theta)$.
   \end{theorem}
\subsection{Harmonic Bundles with Symplectic Structures}
\subsubsection{Results}
We first state the main results of this paper.  Let $X$ be a smooth projective variety over $\mathbb{C}$ and $H\subset X$ be a normal crossing divisor.
\begin{theorem}[Theorem \ref{KH}]\label{intro main}
 The following objects are equivalent on $(X,H)$
    \begin{itemize}
  \item  Good wild harmonic bundles with a symplectic structure.
  \item  Good filtered polystable Higgs bundles equipped with a perfect skew-symmetric pairing satisfying the vanishing condition (\ref{vanish0}).
  \end{itemize}
\end{theorem}
We give an outline of the proof and an explanation of notions in Section \ref{intro non-cpt}. We can regard this result as a Kobayashi-Hitchin correspondence with skew-symmetry. \par
\subsubsection{Compact case}
The contents here are written in Section \ref{cpt sec}.\par
Let $X$ be a compact K\"ahler manifold and $(E,\overline\partial_E,\theta,h)$ be a harmonic bundle of rank $2n$ on $X$.  A symplectic structure $\omega$ for $(E,\overline\partial_E,\theta,h)$ is a non-degenerate skew-symmetric holomorphic pairing of $E$ such that it is skew-symmetric with $\theta$ and compatible with $h$ (See Definition \ref{sym-p}). \par
Let $G$ be a Lie group. We assume $G$ to be semi-simple or algebraic reductive. Let $P\to X$ be a principal $G$-bundle and $\nabla$ be a flat connection of it. We say $\nabla$ is reductive if the corresponding representation $\rho:\pi_1(X)\to G$ is semisimple. Let $K\subset G$ be a maximal compact subgroup. Since $P$ has a flat connection $\nabla$,  a $K$-reduction $P_K$ of $P$ is equivalent to a smooth map $f:\widetilde{X}\to G/K$. Here $\widetilde{X}$ is the universal covering of $X$. We say that the reduction is \textit{harmonic} if $f$ is a harmonic map. It was shown in \cite{Co,Do,S3} that $P$ has a harmonic reduction if and only if $\nabla$ is reductive. 
\par
The first result of this paper is as follows.
\begin{theorem}[Theorem \ref{H-P}]\label{intro main2}
  Let $X$ be a compact K\"ahler manifold. The following objects are equivalent on $X$.
   \begin{itemize}
 \item Polystable Higgs bundle of rank 2n with vanishing Chern classes equipped with a perfect skew-symmetric pairing.
 \item Harmonic bundle of rank 2n equipped with a symplectic structure.
 \item Principal $\mathrm{Sp}(2n,\mathbb{C})$-Bundle with a reductive flat connection.
  \end{itemize}
 \end{theorem}
 The equivalence of the first two objects is a consequence of Theorem \ref{intro main}. In Section \ref{cpt sec} we give the proof of the equivalence of the last two objects. 
\par
We explain the outline of it here. Let $(E,\overline\partial_E,\theta,h)$ be a harmonic bundle of rank $2n$ with a symplectic structure $\omega$. Let $P_E\to X$ be the principal $\mathrm{GL}(2n,\mathbb{C})$-bundle associated to $E$. $\omega$ defines a $\mathrm{Sp}(2n,\mathbb{C})$-reduction $P_{E,\mathrm{Sp}(2n,\mathbb{C})}$ of $P_E$. We show that $\nabla_h$ defines a flat connection $\nabla$ on $P_{E,\mathrm{Sp}(2n,\mathbb{C})}$ and $h$ defines a $\mathrm{Sp}(2n)$-reduction $P_{E,\mathrm{Sp}(2n)}$ of $P_{E,\mathrm{Sp}(2n,\mathbb{C})}$. Here $\mathrm{Sp}(2n)$ is the standard maximal comapct subgroup of $\mathrm{Sp}(2n,\mathbb{C})$. Since $h$ is a pluri-harmonic metric, the induced map $f_h:\widetilde{X}\to\mathrm{Sp}(2n,\mathbb{C})/\mathrm{Sp}(2n)$ is harmonic and hence $\nabla$ is reductive.\par
Before we explain the converse construction, we briefly recall the construction of a harmonic bundle from a principle $G$-bundle with a reductive flat connection $\nabla.$ Let $\tau:G\to GL(V)$ be a linear representation and $E:=P\times^\tau V$ be the associated vector bundle. Let $D$ be the flat connection induced by $\nabla$ and $h$ be the metric induced by $f$. We have a decomposition $D=D_h+\phi$ such that $D_h$ is a metric connection and $\phi$ is self-adjoint w.r.t. $h$. Let $D^{0,1}_h$ and $\phi^{1,0}$ be the $(0,1)$- and $(1,0)$-part of $D_h$ and $\phi$. If $\nabla$ is reductive (i.e. $f$ is harmonic), then $D^{0,1}_h\circ D^{0,1}_h=0$ and $D^{0,1}_h\phi^{1,0}=0$ holds. Hence $(E,D^{0,1}_h,\phi^{1,0},h)$ is a harmonic bundle (See \cite{S3} for details).
\par 
We apply the construction to the case $G=\mathrm{Sp}(2n,\mathbb{C})$. Let $P\to X$ be a principal $\mathrm{Sp}(2n,\mathbb{C})$-bundle with a reductive flat connection $\nabla.$
Let $i:\mathrm{Sp}(2n,\mathbb{C})\to\mathrm{GL}(2n,\mathbb{C})$ be the standard representation of $\mathbb{C}^{2n}$, $E:=P\times^i \mathbb{C}^{2n}$, and $(E,D^{0,1}_h,\phi^{1,0},h)$ be the constructed harmonic bundles. $E$ has a naturally defined smooth skew-symmetric pairing $\omega$. We show that it is holomorphic, skew-symmetric with $\theta$, and compatible with $h$. Hence $(E,D^{0,1}_h,\phi^{1,0},h)$ is a harmonic bundle with a symplectic structure $\omega$.

\subsubsection{Non-Compact case}\label{intro non-cpt}
The contents here are written in Section \ref{f sh0} and \ref{main s}. Here, we explain the outline of the proof of Theorem \ref{intro main}.\par
Let $X$ be a smooth projective variety and $H$ be a normal crossing divisor.  Let $(E,\overline\partial_E, \theta, h)$ be a good wild harmonic bundle on $X-H$ and   $(\mathcal{P}_*\mathcal{V},\theta)$ be a  good filtered Higgs bundle on $(X, H)$.
In the latter half of this paper, we study the good wild harmonic bundles and good filtered Higgs bundles when they admit a symplectic structure and a perfect skew-symmetric pairing.\par
 A perfect skew-symmetric pairing $\omega$ on  $(\mathcal{P}_*\mathcal{V},\theta$) is a morphism of filtered bundle
\begin{equation*}
  \omega:\mathcal{P}_*\mathcal{V}\otimes\mathcal{P}_*\mathcal{V}\to\mathcal{P}^{(0)}_*(\mathcal{O}_X(\ast H))
      \end{equation*}
  such that it is skew-symmetric and  induces an isomorphism $\Psi_\omega:(\mathcal{P}_*\mathcal{V},\theta)\to(\mathcal{P}_*\mathcal{V}^\vee,-\theta^\vee$). See Section 4 for more details on the pairing of filtered bundles.\par
  In section \ref{KH par}, we show that when the good wild harmonic bundle admits a symplectic structure, then the  good filtered Higgs bundle obtained by prolongation admits a perfect skew-symmetric pairing:
 \begin{proposition}[Proposition \ref{w-f}]\label{intro prop1}
 Let $(E,\overline\partial_E,\theta, h)$ be a good wild harmonic bundle equipped with symplectic structure $\omega$. Then $(\mathcal{P}_*^hE, \theta)$ is a  $\mu_L$-polystable
good filtered Higgs bundle equipped with a perfect skew-symmetric pairing $\omega$ and satisfies the vanishing condition (\ref{vanish0}).
 \end{proposition}
 We show that the converse also holds.
 In section \ref{KH part}, we show how a good filtered Higgs bundle decomposes when it admits a perfect skew-symmetric pairing:
 \begin{proposition}[Proposition \ref{B4} and \ref{B5}]\label{intro prop2}
Let $(\mathcal{P}_*\mathcal{V},\theta$) be a $\mu_L$-polystable good filtered Higgs bundle equipped with perfect skew-symmetic pairing $\omega$ and satisfies the vanishing condition (\ref{vanish0}). Then there exist stable Higgs bundles $(\mathcal{P}_*\mathcal{V}_i^{(0)}, \theta_i^{(0)})$ $(i=1,\dots,p(0))$, $(\mathcal{P}_*\mathcal{V}_i^{(1)}, \theta_i^{(1)})$ $(i=1,\dots,p(1))$ and $(\mathcal{P}_*\mathcal{V}_i^{(2)}, \theta_i^{(2)})$ $(i=1,\dots,p(2))$ of degree 0 on $X$ such that the following holds.
\begin{itemize}
\item $(\mathcal{P}_*\mathcal{V}_i^{(0)}, \theta_i^{(0)})$ is equipped with a symmetric pairing $P_i^{(0)}$.
\item  $(\mathcal{P}_*\mathcal{V}_i^{(1)}, \theta_i^{(1)})$ is equipped with a skew-symmetric pairing $P_i^{(1)}$.
\item  $(\mathcal{P}_*\mathcal{V}_i^{(2)}, \theta_i^{(2)})$ $\centernot{\simeq}$ $(\mathcal{P}_*\mathcal{V}_i^{(2)}, -\theta_i^{(2)})^\vee$.
\item There exists positive integers $l(a,i)$ and an isomorphism 
\begin{align*}
(\mathcal{P}_*\mathcal{V}, \theta)\simeq\bigoplus_{i=1}^{p(0)}(\mathcal{P}_*\mathcal{V}_i^{(0)}, \theta_i^{(0)})\otimes\mathbb{C}^{2l(0,i)}\oplus&\bigoplus_{i=1}^{p(1)}(\mathcal{P}_*\mathcal{V}_i^{(1)}, \theta_i^{(1)})\otimes\mathbb{C}^{l(1,i)}\\
&\oplus \bigoplus_{i=1}^{p(2)}\bigg(\big((\mathcal{P}_*\mathcal{V}_i^{(2)}, \theta_i^{(2)})\otimes\mathbb{C}^{l(2,i)}\big)\oplus\big((\mathcal{P}_*\mathcal{V}_i^{(2)}, -\theta_i^{(2)})^\vee\otimes(\mathbb{C}^{l(2,i)})^\vee\big)\bigg).
\end{align*}
Under this isomorphism, $\omega$ is identified with the direct sum of $P_{i}^{(0)}\otimes \omega_{\mathbb{C}^{2l(0,i)}}$, $P_{i}^{(1)}\otimes C_{\mathbb{C}^{l(1,i)}}$ and $\widetilde{\omega}_{(E_i^{(2)}, \theta_i^{(2)})}\otimes C_{\mathbb{C}^{l(2,i)}}$
\item $(\mathcal{P}_*\mathcal{V}_i^{(a)},\theta_i^{(a)})$ $\centernot{\simeq}$ $(\mathcal{P}_*\mathcal{V}_j^{(a)}, \theta_j^{(a)})$ $(i\neq j)$ for a=0,1,2, and $(\mathcal{P}_*\mathcal{V}_i^{(2)}, \theta_i^{(2)})$ $\centernot{\simeq}$ $(\mathcal{P}_*\mathcal{V}_j^{(2)}, -\theta_j^{(2)})^\vee$ for any $i, j.$
\end{itemize}
Moreover, there exists a harmonic metric $h$ on $(\mathcal{V}, \theta)|_{X\setminus D}$ such that $\mathrm{(i)}$ h is adapted to $\mathcal{P}_*\mathcal{V}$,  $\mathrm{(ii)}$ it is compatible with $\omega.$
 \end{proposition}
We give more details on the harmonic metric in Proposition \ref{B5}. Theorem \ref{intro main} is proved by combining Proposition \ref{intro prop1} and \ref{intro prop2}.
\subsection*{Relation to other works}
In \cite{qm1}, Li and Mochizuki studied harmonic bundles with an additional structure called real structure. A real structure is a holomorphic non-degenerate pairing of the given bundle such that the Higgs field is symmetric with it and the harmonic metric is compatible. Although they focused on the study of generically regular semisimple Higgs bundle, they also obtained the Kobayashi-Hitchin correspondence with symmetry.
\begin{theorem}[{\cite[Theorem 3.28]{qm1}}] Let $X$ be a compact Riemann surface and $D\subset X$ be a divisor. Then the following objects are equivalent on $(X,D)$.
\begin{itemize}
\item Wild harmonic bundles on $(X,D)$ with a real structure.
\item Polystable good filtered Higgs bundles of degree 0 equipped with a perfect symmetric pairing.
\end{itemize}
\end{theorem}
Although they only proved for the Riemann surface case, generalization to higher dimensions is straightforward.\par
In \cite{S4}, Simpson established the one-on-one correspondence for reductive flat principal $G$-bundle and semistable $G$-Higgs bundle. Here, we assume $G$ to be a complex reductive algebraic Lie group. Therefore we can regard Theorem \ref{intro main2} as a little bit of a detailed version for $G=\mathrm{Sp}(2n,\mathbb{C})$. 

\subsubsection*{Acknowlegment}
The author thanks his supervisor Hisashi Kasuya for his constant support and many helpful advice. He also helped to improve the readability of this paper.
The author was supported by JST SPRING, Grant Number JPMJSP2138.
\section{Harmonic bundles with symplectic structure}\label{cpt sec}
\subsection{Skew-symmetric pairings of vector spaces}
Let $V$ be a complex vector space of dimension $n$. We fix a hermitian metric $h$ on $V$. Let $V^\vee$ be the dual of $V$. From a hermitian metric $h$ we have an anti-linear map:
\begin{equation*}
\Psi_h:V\to V^\vee
\end{equation*}
defined as $\Psi_h(u)(v):=h(v,u)$ for $u,v\in V.$ \par
 We have an induced hermitian metric $h^\vee$ on $V^\vee$ defined as 
 \begin{equation*}
 h^\vee(u^\vee,v^\vee):=h(\Psi_h^{-1}(v^\vee),\Psi_h^{-1}(u^\vee)).
 \end{equation*}
 Let $\omega$ be a non-degenerate skew-symmetric bilinear form on $V$. We obtain a linear map,
 \begin{equation*}
 \Psi_\omega:V\to V^\vee
  \end{equation*}
 defined as $\Psi_\omega(u)(v):=\omega(u,v)$. \par
 We have an induced skew-symmetric bilinear form $\omega^\vee$ on $V^\vee$ defined as,
\begin{equation*}
 \omega^\vee(u^\vee,v^\vee):=\omega(\Psi_\omega^{-1}(u^\vee),\Psi_\omega^{-1}(v^\vee)).
\end{equation*}
\begin{definition}
Let $(V,h)$ be a vector space with hermitian metric. Let $\omega$ be a non-degenerate skew-symmetric bilinear form on $V$. $\omega$ is compatible with $(V,h)$ if 
\begin{equation*}
\Psi_\omega:(V,h)\to (V^\vee,h^\vee)
\end{equation*}
is an isometry.
\end{definition}
The following Lemma was proved in \cite{qm1} without proof. We give the proof for convenience.
\begin{lemma}\label{m-s}
The following conditions are equivalent
\begin{itemize}
\item h is compatible with $\omega$.
\item $\Psi_{h^\vee}\circ\Psi_{\omega}=\Psi_{\omega^\vee}\circ\Psi_{h}.$
\item $\omega(u,v)=\overline{\omega^\vee(\Psi_h(u),\Psi_h(v))}$ for any $u,v\in V$. 
\end{itemize}
\end{lemma}
\begin{proof}
For a matrix $A$, we denote the transpose of it as $A^T$.
Let $<e_1,\dots,e_n>$ be a basis of $V$ and $<e^\vee_1,\dots,e^\vee_n>$  be the dual basis of $V^\vee$. Let $H:=(h(e_i,e_j))_{1\leq i,j\leq n}$, $\Omega:=(\omega(e_i,e_j))_{1\leq i,j\leq n}.$ The representation matrix of $\Psi_h$ is $H$, $\Psi_{h^\vee}$ is $(H^{-1})^{T}$, $\Psi_\omega$ is $\Omega^{T}$ and $\Psi_{\omega^\vee}$ is $\Omega^{-1}$.\par
When $h$ is compatible with $\omega$, then $(H^{-1})^T=\Omega^{-1} H\overline{\Omega^{-1}}^T$ stands. $\Psi_{h^\vee}\circ\Psi_{\omega}=\Psi_{\omega^\vee}\circ\Psi_{h}$ is equivalent to $(H^{-1})^T\overline{\Omega^T}=\Omega^{-1} H$. The third condition is equivalent to the equality $\Omega^T=\overline{H^{T}\Omega^{-1}H}$. Hence the three conditions are equivalent.
\end{proof}
\subsection{Harmonic bundles with symplectic structure}
Let  $X$ be a complex manifold and $(E,\overline\partial_E,\theta)$ be a Higgs bundle on $X$.
\begin{definition} Let $(E^\vee,\overline\partial_{E^\vee})$ be the dual holomorphic bundle of $(E,\overline\partial_E)$.  A skew-symmetric pairing $\omega$ of $E$ is a global holomorphic section of $E^\vee\otimes E^\vee$
such that $\omega(u,v)=-\omega(v,u)$ holds for any section  $u,v$ of $E$. We say that $\omega$ is perfect if the induced morphism $\Psi_\omega:E\to E^\vee$ is an isomorphism.
\end{definition} 
 We note that when a holomorphic bundle has a perfect symplectic pairing, the rank of it is even.
\begin{definition} A skew-symmetric pairing $\omega$ of the Higgs bundle $(E,\overline\partial_E,\theta)$ is a  skew-symmetric pairing of $(E,\overline\partial_E)$ such that $\omega(\theta\otimes\mathrm{Id})=-\omega(\mathrm{Id}\otimes\theta)$ holds. We call $\omega$ perfect if it is a perfect skew-symmetric pairing of  $(E,\overline\partial_E)$. \end{definition}
A skew-symmetric pairing $\omega$ for  $(E,\overline\partial_E,\theta)$ induces a morphism $\Psi_\omega:(E,\theta)\to (E^\vee,-\theta^\vee)$. Here $\theta^\vee$ is the Higgs field of $E^\vee$ induced from $\theta$.
\begin{remark}
A Higgs bundle with a skew-symmetric pairing is called $\mathrm{Sp}(2n,\mathbb{C})$-Higgs bundle in \cite{GGM}.
\end{remark}
\begin{definition}\label{sym-p}
A symplectic structure $\omega$ of the harmonic bundle $(E,\overline\partial_E,\theta,h)$ is a perfect skew-symmetric pairing of $(E,\overline\partial_E,\theta)$ such that  $h_{|P}$ is compatible with $\omega_{|P}$ for any $P\in X.$ 
 \end{definition}
 \subsection{Harmonic metrics on Principal $G$-bundles}
Let $G$ be a Lie group. In this section, we briefly review harmonic metrics on the principal $G$-bundle. Let $X$ be a Riemannian manifold.
 \begin{definition} Let $P\to X$ be a principal $G$-bundle and $\nabla$ be a flat connection on it. $\nabla$ is called reductive if the corresponding representation $\rho:\pi_1(M)\to G$ is semisimple.
 \end{definition}
 Let $K\subset G$ be a maximal compact subgroup and let $P_K$ be a $K$-reduction of $P$. When $P$ admits a flat connection $\nabla$, to give a  $K$-reduction $P_K$ is equivalent to give a $\pi_1(X)$-equivalent smooth map
 \begin{equation*}
 f:\widetilde{X}\to G/K.
 \end{equation*} 
 Here $\widetilde{X}$ is the universal covering of $X$.\par
 The following result was proved by Donaldson \cite{Do} (when $X$ is a compact Riemann surface and $G=\mathrm{SL}(2,\mathbb{C})$), Corlette \cite{Co} (when $X$ is compact and for semisimple Lie groups) and Simpson \cite{S3} (when $X$ is compact and for algebraic reductive groups).
 \begin{theorem}[\cite{Co,Do,S3}]
Suppose $X$ to be compact. Let $P\to X$ be a principal $G$-bundle with a flat connection $\nabla$. Then there exists a  $\pi_1(X)$-equivalent harmonic map $f:\widetilde{X}\to G/K$ if and only if $\nabla$ is reductive.
 \end{theorem}
From now on, we assume $X$ to be a compact K\"ahler manifold. Let $\pi : G\to \mathrm{GL}(V)$ be a linear representation. We briefly recall how to induce a Higgs bundle structure to $E:= P\times^\pi V$ from a principle $G$-bundle with a reductive flat connection $\nabla$. See \cite{S3} for details. Let $D$ be the induced flat connection of $E$. The harmonic map $f$ induces a metric $h$ on $E$. Let $D=D_h+\phi$ be the decomposition such that $D_h$ is the metric connection and $\phi$ is self-adjoint w.r.t. $h$. Let $D^{0,1}_h$ be the (0,1)-part of $D_h$ and $\theta$ be the (1,0)-part of $\phi$. The harmonicity of $f$ implies that  $D^{0,1}_h\circ D^{0,1}_h=0$ and $D^{0,1}_h\theta=0$. Hence we obtain a harmonic bundle $(E,D^{0,1}_h,\theta,h)$.
 \subsection{Harmonic bundles and  Principal $\mathrm{Sp}(2n,\mathbb{C})$-Bundles}
 Throughout this section, we assume $X$ to be a compact K\"ahler manifold.  In this section, we prove the following:
 \begin{theorem}\label{H-P}
 Let $X$ be a compact K\"ahler manifold. The following objects are equivalent on $X$.
   \begin{itemize}
   \item Polystable Higgs bundle of rank 2n with vanishing Chern classes equipped with a perfect skew-symmetric pairing.
 \item Harmonic bundle of rank 2n equipped with a symplectic structure.
 \item Principal $\mathrm{Sp}(2n,\mathbb{C})$-bundle with a reductive flat connection.
  \end{itemize}
 \end{theorem}
 \begin{proof} The equivalence of the first two objects is a consequence of Corollary \ref{KH cpt}. We give the proof of the equivalence of the last two objects in the end of the section.
  \end{proof}
   To prove Theorem \ref{H-P}, we prepare some Propositions.
 \begin{proposition}\label{one}
 Let $(E,\overline\partial_E)$ be a holomorphic bundle of rank 2n on $X$  and $\omega$ be a perfect skew-symmetric pairing of it.
 Let $P_E\to X$ be the principal $\mathrm{GL}(2n,\mathbb{C})$-bundle associated to $E$. Then $P_E$ has a reduction to $P_{E,\mathrm{Sp}(2n,\mathbb{C})}$ such that $P_{E,\mathrm{Sp}(2n,\mathbb{C})}\to X$ is a principal $\mathrm{Sp}(2n,\mathbb{C})$-bundle.
  \end{proposition}      
 \begin{proof}
To prove the claim, it is enough to prove that there exists an open covering $\{U_i\}_{i\in\Lambda}$ and a family of section $\{(e_{k,i})^{2n}_{k=1}\}_{i\in\Lambda}$ of $E$ such that
 \begin{itemize}
 \item $(e_{k,i})^{2n}_{k=1}$ is a frame of $E$ on $U_i$,
 \item The family of transition function $\{g_{ij}\}_{i,j\in\Lambda}$ associated to $\{(e_{k,i})^{2n}_{k=1}\}_{i\in\Lambda}$ takes value in $\mathrm{Sp}(2n,\mathbb{C})$.
  \end{itemize}
  To show such an open covering and frames exists, we only have to show that there exists an open covering $\{U_i\}_{i\in\Lambda}$ of $X$ and on each $U_i$, we have a  frame $(e_{k,i})^{2n}_{k=1}$ of $E$  such that w.r.t $(e_{k,i})^{2n}_{k=1}$, $\omega|_{U_i}$ has the form
  \begin{equation*}
\omega|_{U_i}=\sum^{n}_{k=1}\big(e^\vee_{k,i}\otimes e^\vee_{k+n,i}-e^\vee_{k+n,i}\otimes e^\vee_{k,i}\big).
  \end{equation*} 
  Here, $e^\vee_{k,i}$ is the dual frame of $e_{k,i}$.
We note that 
 \begin{equation*}
 \begin{pmatrix} 
  \omega|_{U_i}(e_{1,i},e_{1,i}) & \dots  & \omega|_{U_i}(e_{1,i},e_{2n,i}) \\
    \vdots &  \ddots & \vdots \\
 \omega|_{U_i}(e_{2n,i},e_{1,i})  & \dots  & \omega|_{U_i}(e_{2n,i},e_{2n,i})
\end{pmatrix}
=
\mathcal{J}_n:=
 \begin{pmatrix} 
 O  & \huge{I_n}  \\
 \huge{-I_n}  & O
\end{pmatrix}.
 \end{equation*}
Here $I_n$ is the $n\times n$ identity matrix. Once we showed such frames exist, then the transition functions obiously take value in $\mathrm{Sp}(2n,\mathbb{C})$. \par
We now prove that such frames exist around any $P\in X$. Let $U_P$ be an open neighborhood of $P$ and $(e_{k})^{2n}_{k=1}$ be a  frame of $E$ on $U_P$. Since $\omega$ is perfect, there exists a $e_k(k\neq 1)$ such that $\omega(e_1,e_k)|_{P}\neq 0$. We may shrink $U_P$ so that $\omega(e_1,e_k)$ does not take $0$ in $U_P$. We may also permute $(e_{k})^{2n}_{k=1}$ so we can assume $\omega(e_1,e_{n+1})$ does not take $0$ in $U_P$. Under this assumption, we construct a new frame $(e'_{k})^{2n}_{k=1}$ as
\begin{align*}
e'_1&:=e_1,\\
e'_{n+1}&:=-\frac{e_1}{\omega(e_1,e_{n+1})},\\
e'_{k}:=e_k-\omega(e_k,e'_{n+1})e&'_1+\omega(e_k,e'_1)e'_{n+1}(k: otherwise).
\end{align*}
By direct calculation, we can check $\omega(e'_1,e'_{n+1})=1$ and $\omega(e'_k,e'_1)=\omega(e'_k,e'_{n+1})=0 (k\neq 1, n+1)$. It is easy to see that $(e'_{k})^{2n}_{k=1}$ is actually a frame. \par
By the same argument as above for $e'_{2}$, we can assume that $\omega(e'_2, e'_{n+2})$  does not take 0 in $U_P$. We construct a new frame $(e^{''}_{k})^{2n}_{k=1}$ as
\begin{align*}
e^{''}_1&:=e'_1,\\
e^{''}_{n+1}&:=e^{'}_{n+1},\\
e^{''}_2&:=e'_2,\\
e^{''}_{n+2}&:=-\frac{e'_2}{\omega(e'_2,e'_{n+2})},\\
e^{''}_{k}:=e'_k-\omega(e'_k,e^{''}_{n+2})e&^{''}_2+\omega(e'_k,e^{''}_2)e^{''}_{n+2}(k:otherwise).
\end{align*}
By direct calculation, we can check $\omega(e^{''}_i,e^{''}_{n+i})=1(i=1,2)$ and $\omega(e^{''}_k,e^{''}_i)=\omega(e^{''}_k,e^{''}_{n+i})=0 (i=1,2, k\neq 1,2, n+1,n+2)$. Continuing this procedure, we finally  obtain a frame $(\widetilde{e}_{k})^{2n}_{k=1}$ on $U_P$ such that 
 \begin{equation*}
 \begin{pmatrix} 
  \omega|_{U_P}(\widetilde{e}_{1},\widetilde{e}_{1}) & \dots  & \omega|_{U_P}(\widetilde{e}_{1},\widetilde{e}_{2n}) \\
    \vdots &  \ddots & \vdots \\
 \omega|_{U_P}(\widetilde{e}_{2n},\widetilde{e}_{1})  & \dots  & \omega|_{U_P}(\widetilde{e}_{2n},\widetilde{e}_{2n})
\end{pmatrix}
=\mathcal{J}_n.
 \end{equation*}
 We can construct such a frame around for arbitrary $P\in X$. Hence we proved the claim.
\end{proof}
We set $\mathrm{Sp}(2n):=\mathrm{Sp}(2n,\mathbb{C})\cap\mathrm{U}(2n)$. Here $\mathrm{U}(2n)$ is the set of unitary matrices. $\mathrm{Sp}(2n)$ is a maximal compact subgroup of $\mathrm{Sp}(2n,\mathbb{C}).$
   \begin{proposition}\label{two}
 Let $(E,\overline\partial_E,\theta,h)$ be a harmonic bundle of rank 2n on $X$ and $\omega$ be a symplectic structure of it. Then the associated principal $\mathrm{Sp}(2n,\mathbb{C})$-bundle $P_{E,\mathrm{Sp}(2n,\mathbb{C})}$ admits a reductive flat connection $\nabla$.
  \end{proposition}  
  \begin{proof} Since $h$ is a pluri-harmonic metric, the connection $\nabla_h=\partial_h+\overline\partial_E+\theta+\theta^\dagger_h$ is a flat connection. Let $\{U_i\}_{i\in\Lambda}$ and $\{(e_{k,i})^{2n}_{k=1}\}_{i\in\Lambda}$ be the open cover and the frame which we constructed in Proposition \ref{one}. Let $\mathfrak{sp}(2n,\mathbb{C})$ be the Lie algebra of $\mathrm{Sp}(2n,\mathbb{C})$. To prove the claim, first, we show that the connection form of $\nabla_h$ w.r.t $(e_{k,i})^{2n}_{k=1}$ is a $\mathfrak{sp}(2n,\mathbb{C})$-valued 1-form on $U_i$. Once this is shown, since the transition functions of $\{(e_{k,i})^{2n}_{k=1}\}_{i\in\Lambda}$ take value in  $\mathrm{Sp}(2n,\mathbb{C})$,  we obtain a connection form on $P_{E,\mathrm{Sp}(2n,\mathbb{C})}$ and hence it induces a connection $\nabla$. The flatness of $\nabla$ follows from the flatness of $\nabla_h$. Reductiveness of $\nabla$ follows from $h$: From  Lemma \ref{w-k}, we know that $h$ defines a $\mathrm{Sp}(2n)$-reduction of $P_{E,\mathrm{Sp}(2n,\mathbb{C})}$. Since $\nabla$ is flat, $h$ induces a map $f_h:\widetilde{X}\to\mathrm{Sp}(2n,\mathbb{C})/\mathrm{Sp}(2n)$. $f_h$ is harmonic since $h$ is a pluri-harmonic metric. Reductiveness of $\nabla$ follows immediately. \par
  Let $A_i$ be the connection form of $\nabla_h$ w.r.t. $(e_{k,i})^{2n}_{k=1}$. Let $h_i$ be a $n\times n$ matrix such that 
  \begin{equation*}
  h_i:=
   \begin{pmatrix} 
 h|_{U_i}(e_{1,i},e_{1,i}) & \dots  &  h|_{U_i}(e_{1,i},e_{2n,i}) \\
    \vdots &  \ddots & \vdots \\
  h|_{U_i}(e_{2n,i},e_{1,i})  & \dots  &  h|_{U_i}(e_{2n,i},e_{2n,i})
\end{pmatrix}.
  \end{equation*}
 From the standard argument of the connections, we have 
 \begin{equation*}
 A_i=h_i^{-1}\partial h_i + \theta|_{U_i} + \theta^\dagger_h|_{U_i}=h_i^{-1}\partial h_i + \theta|_{U_i} + h_i^{-1} \overline{\theta^T}|_{U_i} h_i.
 \end{equation*}  
 We show that $A_i$ takes value in $\mathfrak{sp}(2n,\mathbb{C})$. First, we show that $\theta|_{U_i}$ takes value in $\mathfrak{sp}(2n,\mathbb{C})$. Recall that the local description of $\omega$ w.r.t. $(e_{k,i})^{2n}_{k=1}$ is $\mathcal{J}_n$. Since $\omega(\theta\otimes\mathrm{Id})=-\omega(\mathrm{Id}\otimes\theta)$ holds,
 \begin{equation*}
 \theta^T|_{U_i}\mathcal{J}_n=-\mathcal{J}_n\theta|_{U_i}
 \end{equation*}
 holds. Hence we showed it.\par
 We next prove $h_i$ takes value in $\mathrm{Sp}(2n,\mathbb{C})$. Once this is shown, then it is obvious that $\theta^\dagger_h|_{U_i}= h_i^{-1} \overline{\theta^T}|_{U_i} h_i$ takes value in  $\mathfrak{sp}(2n,\mathbb{C})$. We also can show that $h^{-1}_i\partial h_i$ takes value in it: Suppose $h_i$ takes value in $\mathrm{Sp}(2n,\mathbb{C})$. Then we have the following
 \begin{equation*}
 h_i^T\mathcal{J}_nh_i=\mathcal{J}_n.
 \end{equation*}
 Then we have 
  \begin{align*}
  h_i^T&=-\mathcal{J}_nh^{-1}_i\mathcal{J}_n,\\
h_i&=-\mathcal{J}_n(h^{-1}_i)^T\mathcal{J}_n,\\
     \partial h_i^T\mathcal{J}_n&h_i+h^T_i\mathcal{J}_n\partial h_i=0.
 \end{align*}
 Hence we have 
 \begin{align*}
 0&= \partial h^T_i\mathcal{J}_nh_i+h^T_i\mathcal{J}_n\partial h_i\\
 &=\partial h^T_i\mathcal{J}_n(-\mathcal{J}_n(h^{-1}_i)^T\mathcal{J}_n)+(-\mathcal{J}_nh^{-1}_i\mathcal{J}_n)\mathcal{J}_n\partial h_i\\
 &=\partial h^T_i(h^{-1}_i)^T\mathcal{J}_n+\mathcal{J}_nh^{-1}_i\partial h_i.
    \end{align*}
 Since $(h^{-1}_i\partial h_i)^T=\partial h^T_i(h^{-1}_i)^T$, $h^{-1}_i\partial h_i$   takes value in $\mathfrak{sp}(2n,\mathbb{C})$. We now prove $h_i$ takes value in $\mathrm{Sp}(2n,\mathbb{C})$. Let $(e^\vee_{k,i})^{2n}_{k=1}$  be the dual frame of $(e_{k,i})^{2n}_{k=1}$, $h^\vee$ be the dual metric of $h$, and $\omega^\vee$ be the dual of $\omega$. Then the matrix realizations of $h^\vee$ w.r.t to $(e^\vee_{k,i})^{2n}_{k=1}$ is $(h^{-1}_i)^T$. Since $\omega$ is compatible with $h$ we can use Lemma \ref{m-s} and hence we have
 \begin{equation*}
 (h_i^{-1})^T=\mathcal{J}_nh_i\mathcal{J}_n^T.
  \end{equation*}
  Hence we have 
  \begin{equation*}
 \mathcal{J}_n=h_i^T\mathcal{J}_nh_i.
  \end{equation*}  
  This shows that  $h_i$ takes value in $\mathrm{Sp}(2n,\mathbb{C})$.
   \end{proof}
   Let $M(2n,\mathbb{C})$ be the set of $2n\times 2n$-matrix, $\mathfrak{p}\subset M(2n,\mathbb{C})$ be the set of hermitian matrix, and $\mathfrak{p}_+\subset \mathfrak{p}$ be the set of positive definite ones. As it is well known the standard exponential map
  \begin{equation*}
  \mathrm{exp}:\mathfrak{p}\to\mathfrak{p}_+
     \end{equation*}
     is a real analytic isomorphism. We set $\mathrm{log}:=(\mathrm{exp})^{-1}$.\par
     Although the following Lemma might be well known to experts, we give the proof for convenience.
   \begin{lemma}\label{w-k}
    Let $E$ be a complex vector bundle, $h$ be a hermitian metric, and  $\omega$ be a smooth perfect skew-symmetric structure. We assume $h$ is compatible with $\omega$. 
   Under this assumption, $h$ defines a $\mathrm{Sp}(2n)$-reduction $P_{E,\mathrm{Sp}(2n)}$ of $P_{E,\mathrm{Sp}(2n,\mathbb{C})}$. 
      \end{lemma}
   \begin{proof}
In Proposition \ref{one}, we constructed an open cover $\{U_i\}_{i\in\Lambda}$ and a family of frame $\{(e_{k,i})^{2n}_{k=1}\}_{i\in\Lambda}$ such that  its transition functions $\{g_{ij}\}_{i,j\in\Lambda}$ takes value in $\mathrm{Sp}(2n,\mathbb{C})$. We recall that $\{g_{ij}\}_{i,j\in\Lambda}$ constructs $P_{E,\mathrm{Sp}(2n,\mathbb{C})}$. To prove $h$ induces a $\mathrm{Sp}(2n)$-reduction, it is enough to show that on each $U_i$, $h$ defines a function $s_i:U_i\to\mathrm{Sp}(2n,\mathbb{C})$ such that if $U_i\cap U_j\neq\emptyset$
   \begin{equation*}
   s^{-1}_i(x)g_{ij}(x)s_j(x)\in \mathrm{Sp}(2n), x\in U_i\cap U_j
            \end{equation*}
   holds. Actually, if we set $g'_{ij}=s^{-1}_ig_{ij}s_j$, then it is easy to check that $\{g'_{ij}\}_{i,j\in\Lambda}$ defines a principal  $\mathrm{Sp}(2n)$-bundle which is a reduction of  $P_{E,\mathrm{Sp}(2n,\mathbb{C})}$.\par
   We now construct $s_i$. Let  $h_i$ be the matrix realization of $h$ w.r.t. $(e_{k,i})^{2n}_{k=1}$ as in Proposition \ref{two}. We showed that $h_i$ takes value in $\mathrm{Sp}(2n,\mathbb{C})$. We set
   \begin{equation*}
   s_i:=\mathrm{exp}\bigg(\frac{\mathrm{log}h_i}{2}\bigg).
   \end{equation*}
 $\mathrm{log}h_i$ makes sense since $h_i$ is a positive definite hermitian matrix. Since $h_i$ takes value in $\mathrm{Sp}(2n,\mathbb{C})$, $\mathrm{log}h_i$ takes value in $\mathfrak{sp}(2n,\mathbb{C})$. Hence  $s_i$ is a $\mathrm{Sp}(2n,\mathbb{C})$-valued smooth function on $U_i$.\par
 We next show that $s_i^{-1}g_{ij}s_j\in\mathrm{U}(n)$. We show this by direct calculation. Before going to the calculation we note that if $U_i\cap U_i\neq\emptyset$, then $h_i=g_{ij}h_j\overline g_{ij}^T$.
 \begin{align*}
\overline{s^{-1}_ig_{ij}s_j}^Ts^{-1}_ig_{ij}s_j&=s_j \overline{g_{ij}}^T s^{-1}_i s^{-1}_i g_{ij} s_j \\
 &=s_j\overline{g_{ij}}^T\mathrm{exp}\bigg(-\frac{\mathrm{log}h_i}{2}\bigg)\mathrm{exp}\bigg(-\frac{\mathrm{log}h_i}{2}\bigg)g_{ij}s_j\\
 &=s_j\overline{g_{ij}}^Th_i^{-1}g_{ij}s_j\\
 &=s_jh^{-1}_js_j\\
 &=s_j\mathrm{exp}\bigg(-\frac{\mathrm{log}h_j}{2}\bigg)\mathrm{exp}\bigg(-\frac{\mathrm{log}h_j}{2}\bigg)s_j\\
&=I_n.\\
  \end{align*}
  The first equation holds since $h_i$ is hermitian. Since $s_i$ is $\mathrm{Sp}(2n,\mathbb{C})$-valued, $s_i^{-1}g_{ij}s_j$ takes value in $\mathrm{Sp}(2n)$. The claim is proved.
          \end{proof}
   Let $i:\mathrm{Sp}(2n,\mathbb{C})\to\mathrm{GL}(2n,\mathbb{C})$ be the standard representaion of $\mathbb{C}^{2n}$.
 \begin{proposition}\label{three}
 Let $P\to X$ be a principal $\mathrm{Sp}(2n,\mathbb{C})$-bundle. Then the associated bundle $E:=P\times^i\mathbb{C}^{2n}$ admits a smooth perfect skew-symmetric pairing $\omega$. 
 \end{proposition}  
 \begin{proof}
 By the definition of $E$, we have an open covering $\{U_i\}_{i\in\Lambda}$ of $X$ and on each $U_i$, we have a   frame $(e_{k,i})^{2n}_{k=1}$ of $E$ such that the associated tranisition functions $\{g_{ij}\}_{i,j\in\Lambda}$ takes value in  $\mathrm{Sp}(2n,\mathbb{C})$. We define a section $\omega_i$ of $E^\vee\otimes E^\vee|_{U_i}$ as 
 \begin{equation*}
 \omega_i:=\sum^{n}_{k=1}\big(e^\vee_{k,i}\otimes e^\vee_{k+n,i}-e^\vee_{k+n,i}\otimes e^\vee_{k,i}\big).
 \end{equation*}
 Here, $e^\vee_{k,i}$ is the dual frame of $e_{k,i}$. We note that 
 \begin{equation*}
 \begin{pmatrix} 
  \omega_i(e_{1,i},e_{1,i}) & \dots  & \omega_i(e_{1,i},e_{2n,i}) \\
    \vdots &  \ddots & \vdots \\
 \omega_i(e_{2n,i},e_{1,i})  & \dots  & \omega_i(e_{2n,i},e_{2n,i})
\end{pmatrix}
=\mathcal{J}_n.
 \end{equation*}
 Since the transition function $\{g_{ij}\}_{i,j\in\Lambda}$ takes value in  $\mathrm{Sp}(2n,\mathbb{C})$, $\omega_{i}|_{U_i\cap U_j}=\omega_{j}|_{U_i\cap U_j}$ holds. Hence we can glue them and construct a global section $\omega$ of $E^\vee\otimes E^\vee$ such that $\omega|_{U_i}=\omega_i$. By the local description of $\omega$, it is a smooth perfect skew-symmetric pairing.
  \end{proof}

 \begin{proposition}\label{four}
 Let $P\to X$ a principle $\mathrm{Sp}(2n,\mathbb{C})$-bundle with a reductive flat connection $\nabla$. Then we obtain a harmonic bundle $(E,\overline\partial_E,\theta,h)$ and it has a symplectic structure $\omega$.
  \end{proposition}  
  \begin{proof}
 By the previous proposition, we have a smooth bundle $E$ with a smooth perfect skew-symmetric pairing $\omega$. Since $\nabla$ is a reductive a flat bundle, we have a  $\pi_1(X)$-equivalent harmonic map $f:\widetilde{X}\to \mathrm{Sp}(2n,\mathbb{C})/\mathrm{Sp}(2n)$. $f$ induces a hermitian metric $h$ on $E$ and by construction, it is compatible with $\omega$.\par
  Let $D_{\nabla}$ be the flat bundle of $E$ induced by $\nabla$. We have a decomposition $D_{\nabla}=D_h+\phi$ such that $D_h$ is a metric connection and $\phi$ is self-adjoint w.r.t. $h$. Let $\theta$ be the $(1,0)$-part of $\phi$.  Since $\phi$ is self-adjoint we have the decomposition $\phi=\theta+\theta^{\dagger}_h$. As we recalled in the previous section, the reductiveness of  $\nabla$ implies that $D^{0,1}_h\circ D^{0,1}_{h}=0$ and $D^{0,1}_h\theta=0$. Hence $(E,D^{0,1}_{h},\theta,h)$ is a harmonic bundle.\par
 Next, we show that $\theta$ is compatible with $\omega$. Let $(e_{k,i})^{2n}_{k=1}$ be the frame that we used in the last proposition, and let $A_i$ be the connection matrix of $D_\nabla$ w.r.t. $(e_{k,i})^{2n}_{k=1}$ (i.e. $D_\nabla=d+A_i$ locally).  Note that $A_i$ takes value in $\mathfrak{sp}(2n,\mathbb{C})$. We briefly recall how we obtain the decomposition $D_{\nabla}=D_h+\phi.$ Let $D^{1,0}$ (resp. $D^{0,1})$ be the (1,0) (resp. (0,1))-part of $D_\nabla$. Let $\delta^{1,0}$ (resp. $\delta^{0,1})$ be the (1,0) (resp. (0,1))-type of the differential operator which makes $D^{1,0}+\delta^{0,1}$ and $D^{0,1}+\delta^{1,0}$ metric connections.  $D_h$ and $\phi$ were defined as follows
 \begin{equation*}
D_h:=\frac{D^{1,0}+D^{0,1}+\delta^{1,0}+\delta^{0,1}}{2},\phi:=\frac{D^{1,0}+D^{0,1}-\delta^{1,0}-\delta^{0,1}}{2}.
 \end{equation*}
We note that $\delta^{1,0}$ and $\delta^{0,1}$do exsits and  locally they are expressed as
\begin{align*}
\delta^{1,0}&=\partial-(A_i^{0,1})^{\dagger}_h+h^{-1}_i\partial h_i,\\
\delta^{0,1}&=\overline\partial-(A_i^{1,0})^{\dagger}_h+h^{-1}_i\overline\partial h_i.
\end{align*}
Hence $\theta$ has the form 
\begin{equation*}
\theta=\frac{A_i^{1,0}-(A_i^{0,1})^{\dagger}_h+h^{-1}_i\partial h_i}{2}.
\end{equation*}
Hence $\theta$ takes value in $\mathfrak{sp}(2n,\mathbb{C})$ and therefore it is compatible with $\omega$. \par
We next prove $\omega$ is holomorphic and hence it is a symplectic structure of $(E,D^{0,1}_{h},\theta,h)$. We have to show $D^{0,1}_h\omega=0$. By the construction of  $D_h$ we have 
\begin{equation*}
D^{0,1}_h=\frac{D^{0,1}+\delta^{0,1}}{2}.
\end{equation*}
Let $B_i$ be the connection matrix of $D^{0,1}_h$. From the local description of $\delta^{0,1}$, $B_i$ is a (0,1)-form which takes value in $\mathfrak{sp}(2n,\mathbb{C})$. Hence we have 
\begin{align*}
D^{0,1}_h\omega&=\overline\partial\mathcal{J}_n-B_i^T\mathcal{J}_n-\mathcal{J}_nB_j=0.
\end{align*}
The first equality follows from the standard argument of connection (See \cite{Ko}, for example). Therefore we proved the claim.
  \end{proof}
  \begin{proof}[Proof of Theorem \ref{H-P}.]
  Proposition \ref{one} and \ref{two} gives a path from a harmonic bundle with symplectic structure to a principal $\mathrm{Sp}(2n,\mathbb{C})$-bundle with a reductive flat connection. The inverse path is given by Proposition \ref{three} and \ref{four}. 
    \end{proof}
\section{Good filtered Higgs bundles and Good Wild Harmonic bunldes}\label{f sh0}
\subsection{Filtered sheaves}\label{f sh}
Let $X$ be a complex manifold and $H$ be a simple normal crossing hypersurface of $X$. Let $H:=\bigcup_{i\in\Lambda}H_\lambda$ be the decomposition such that each $H_i$ is smooth.
\subsubsection{Filtered sheaves}
For any $P\in H$, a holomorphic coordinate neighborhood $(U_P,z_1,\dots,z_n)$ around $P$ is called admissible if $H_P:=H\cap U_P=\bigcup^{l(P)}_{i=1}\{z_i=0\}$. For admissible coordinate neighborhood, we obtain a map $\rho_P:\{1,\dots,l(P)\}\to\Lambda$ such that $H_{\rho_{P(i)}}\cap U_p=\{z_i=0\}.$ We also obtain a map $\kappa_P:\mathbb{R}^{\Lambda}\to\mathbb{R}^{l(P)}$ by $\kappa_P(\textbf{a})=(a_{\rho(1)},\dots,a_{\rho(l(P))})$.\par
Let $\mathcal{O}_{X}(\ast H)$ be the sheaf of meromorphic function on $X$ which may have poles along $H$. Let $\mathcal{V}$ be a torsion free $\mathcal{O}_{X}(\ast H)$-module. A filtered sheaf over $\mathcal{V}$ is defined to be a tuple of coherent $\mathcal{O}_X$-submodules $\mathcal{P}_{\textbf{\textit{a}}}\mathcal{V}\subset \mathcal{V}$ ($\textbf{\textit{a}}\in\mathbb{R}^{\Lambda}$) such that 
\begin{itemize}
\item  $\mathcal{P}_{\textbf{\textit{a}}}\mathcal{E}\subset \mathcal{P}_{\textbf{\textit{b}}}\mathcal{E}$ if $\textbf{\textit{\textit{a}}}\leq\textbf{\textit{b}}$, i.e. $a_i\leq b_i$ for any $i\in\Lambda$.
\item $\mathcal{P}_{\textbf{\textit{a}}}\mathcal{E}\otimes\mathcal{O}_X(\ast H)=\mathcal{E}$ for any $\textbf{\textit{a}}\in\mathbb{R}^{\Lambda}$.
\item $\mathcal{P}_{\textbf{\textit{a}+\textit{n}}}\mathcal{E}=\mathcal{P}_{\textbf{\textit{a}}}\mathcal{E}\otimes\mathcal{O}_X(\sum_{i\in\Lambda}n_i H_i)$ for any $\textbf{\textit{a}}\in\mathbb{R}^{\Lambda}$ and for any $\textbf{\textit{n}}\in\mathbb{Z}^{\Lambda}$.
\item For any $\textbf{\textit{a}}\in\mathbb{R}^{\Lambda}$, there exists $\bm{\epsilon}$ $\in\mathbb{R}_{>0}^{\Lambda}$ such that $\mathcal{P}_{\textbf{\textit{a}}+\bm{\epsilon}}\mathcal{E}=\mathcal{P}_{\textbf{\textit{a}}}\mathcal{E}$.
\item For any $P\in H$, let $(U_P,z_1,\dots,z_n)$ be an admissible coordinate of $P$. Then $\mathcal{P}_{\textbf{\textit{a}}}\mathcal{E}|_{U_P}$ depends only on $\kappa_P(\textbf{\textit{a}})$ for any $\textbf{\textit{a}}\in\mathbb{R}^{\Lambda}.$
\end{itemize}
For any coherent $\mathcal{O}_{X}(\ast H)$-submodule $\mathcal{E'}\subset\mathcal{E}$, we obtain a filtered sheaf $\mathcal{P}_{*}\mathcal{E'}$ over $\mathcal{E}'$ by $\mathcal{P}_{\textbf{\textit{a}}}\mathcal{E}'=\mathcal{P}_{\textbf{\textit{a}}}\mathcal{E}\cap \mathcal{E}'$. If $\mathcal{V}'$ is saturated, i.e. $\mathcal{E}'':=\mathcal{E}/\mathcal{E}'$ is torsion-free, then we obtain a filtered sheaf $\mathcal{P}_{*}\mathcal{E}''$ over $\mathcal{E}''$ by $\mathcal{P}_{\textbf{\textit{a}}}\mathcal{E}'':=\mathrm{Im}(\mathcal{P}_{\textbf{\textit{a}}}\mathcal{E}\to\mathcal{E}'')$.\par
A morphism of filtered sheaves $f:\mathcal{P}_{*}\mathcal{E}_1\to\mathcal{P}_{*}\mathcal{E}_2$ is a morphism of $\mathcal{O}_{X}(\ast H)$-modules such that $f(\mathcal{P}_{\textbf{\textit{a}}}\mathcal{E}_1)\subset\mathcal{P}_{\textbf{\textit{a}}}\mathcal{E}_2$ for any $\textbf{\textit{a}}\in\mathbb{R}^{\Lambda}$. \par
Let $\mathcal{P}_*\mathcal{E}$ be a flltered sheaf on $X$. For every open subset $U\subset X$, we can induce a filtered sheaf over $\mathcal{E}|_U$ from $\mathcal{P}_*\mathcal{E}$. We denote this filtered sheaf $\mathcal{P}_*\mathcal{E}|_U$. Conversely, let $X=\bigcup_{i\in\Lambda}U_i$ be an open covering. Let $\mathcal{P}_*\mathcal{V}_i$ be a   filtered sheaf on $U_i$. If $\mathcal{P}_*\mathcal{E}_i|_{U_i\cap U_j}=\mathcal{P}_*\mathcal{E}_i|_{U_i\cap U_j}$, we have a unique filtered sheaf $\mathcal{P}_*\mathcal{E}$ on $X$ such that $\mathcal{P}_*\mathcal{E}|_{U_i}=\mathcal{P}_*\mathcal{E}_i.$ See \cite[Section 2.1.2]{M2} for details of this paragraph.
\subsubsection{Filtered Higgs sheaves}
Let $\mathcal{E}$ be a torsion-free coherent $\mathcal{O}_{X}(\ast H)$-module. A Higgs field $\theta:\mathcal{V}\to\Omega^1_X\otimes\mathcal{V}$ is a $\mathcal{O}_X$-linear morphism of sheaves such that $\theta\wedge\theta=0$. When $\mathcal{V}$ is equipped with a Higgs field, a sub-Higgs sheaf of  $\mathcal{V}'$ is a coherent $\mathcal{O}_X(\ast H)$-submodule $\mathcal{V}'\subset\mathcal{V}$ such that $\theta(\mathcal{V}')\subset\Omega^1_X\otimes\mathcal{V}'$.  A pair of a filtered sheaf $\mathcal{P}_*\mathcal{V}$ over $\mathcal{V}$ and a Higgs field $\theta$ of $\mathcal{V}$ is called a filtered Higgs bundle.
\subsection{$\mu_L$- stability condition for  filtered Higgs sheaves}
Throughout this section, we assume $X$ to be a smooth projective variety, $H=\bigcup_{i\in\Lambda}H_i$ to be a normal crossing divisor of it, and $L$ to be an ample line bundle.
\subsubsection{Slope of filtered sheaves}\label{slope}
Let $\mathcal{P}_*\mathcal{E}$ be a filtered sheaf on $(X, H)$. We recall the definition of the first Chern class $c_1(\mathcal{P}_*\mathcal{E})$. Let $\textbf{\textit{a}}\in\mathbb{R}^{\Lambda}$. Let $\eta_i$ be a generic point on $H_i$. The $\mathcal{O}_{X,\eta_i}$-module $(\mathcal{P}_{\textbf{\textit{a}}}\mathcal{E})_{\eta_i}$ only depends on $a_i$ which we denote as $\mathcal{P}_{a_i}(\mathcal{E}_{\eta_i})$. We obtain a $\mathcal{O}_{H_i,\eta_i}$-module $\mathrm{Gr}^{\mathcal{P}}_{a_i}(\mathcal{E}_{\eta_i}):=\mathcal{P}_{a_i}(\mathcal{E}_{\eta_i})/\sum_{b_i<a_i}\mathcal{P}_{b_i}(\mathcal{E}_{\eta_i})$. $c_1(\mathcal{P}_*\mathcal{E})$ is defined as 
\begin{equation*}
c_1(\mathcal{P}_*\mathcal{E}):=c_1(\mathcal{P}_{\textbf{\textit{a}}}\mathcal{E})-\sum_{i\in\Lambda}\sum_{a_i-1<a\leq a_i}a\cdot \mathrm{rank}\mathrm{Gr}^{\mathcal{P}}_{a}(\mathcal{E}_{\eta_i})\cdot[H_i].
\end{equation*}
Here, $[H_i]\in H^2(X,\mathbb{R})$ is the cohomology class induced by $H_i$. \par
The slope $\mu_L(\mathcal{P}_*\mathcal{E})$ of a filtered sheaf $\mathcal{P}_*\mathcal{E}$ with respect to $L$ is  defined as 
\begin{equation*}
\mu_L(\mathcal{P}_*\mathcal{E})=\frac{1}{\mathrm{rank}\mathcal{E}}\int_Xc_1(\mathcal{P}_*\mathcal{E})\cdot c_1(L)^{\mathrm{dim}X-1}.
\end{equation*}
\subsubsection{$\mu_L$-stablity condition}
Let $(\mathcal{P}_*\mathcal{E},\theta)$ be a filtered Higgs bundle over $(X,H)$. We say that $(\mathcal{P}_*\mathcal{E},\theta)$ is $\mu_L$-stable (resp. $\mu_L$-semistable) if for every sub Higgs sheaf $\mathcal{E}'\subset\mathcal{E}$ such that $0<\mathrm{rank}\mathcal{E}'<\mathrm{rank}\mathcal{E}$, $\mu_L(\mathcal{P}_*\mathcal{E}')<\mu_L(\mathcal{P}_*\mathcal{E})$ (resp. $\mu_L(\mathcal{P}_*\mathcal{E}')\leq\mu_L(\mathcal{P}_*\mathcal{E}$) ) holds.\par
We say that $(\mathcal{P}_*\mathcal{E},\theta)$ is $\mu_L$-polystable if the following two conditions are satisfied 
\begin{itemize}
\item $(\mathcal{P}_*\mathcal{E},\theta)$ is $\mu_L$-semistable.
\item We have a decomposition $(\mathcal{P}_*\mathcal{E},\theta)=\bigoplus_i (\mathcal{P}_*\mathcal{E}_i,\theta_i)$ such that each $(\mathcal{P}_*\mathcal{E}_i,\theta_i)$ is $\mu_L$-stable and $\mu_L(\mathcal{P}_*\mathcal{E})=\mu_L(\mathcal{P}_*\mathcal{E}_i)$ holds.
\end{itemize}
\subsubsection{Canonical decomposition}
Let $(\mathcal{P}_*\mathcal{E}_1,\theta_1)$ and $(\mathcal{P}_*\mathcal{E}_2,\theta_2)$ be filtered Higgs bundle on $(X,H)$. We use the following result frequently without mention.
\begin{proposition}[{\cite[Lemma 3.10]{M1}}]
Let $(\mathcal{P}_*\mathcal{E}_i,\theta_i)(i=1,2)$ be $\mu_L$-semistable reflexive satrated Higgs sheaves such that $\mu_L(\mathcal{E}_1)=\mu_L(\mathcal{E}_2)$. Assume either one of the following:
\begin{itemize}
\item One of $(\mathcal{P}_*\mathcal{E}_i,\theta_i)$ is $\mu_L$-stable and $\mathrm{rank}\mathcal{E}_1=\mathrm{rank}\mathcal{E}_2$ holds.
\item Both $(\mathcal{P}_*\mathcal{E}_i,\theta_i)$ are $\mu_L$-stable.
\end{itemize}
If there is a non-trivial map $f:(\mathcal{P}_*\mathcal{E}_1,\theta_1)\to(\mathcal{P}_*\mathcal{E}_2,\theta_2)$, then $f$ is an isomorphism.
\end{proposition}
The following is straightforward from the above result.
\begin{corollary}Let $(\mathcal{P}_*\mathcal{E},\theta)$ be a $\mu_L$-polystable reflexive satrated Higgs sheaves. Then there exists an unique decomposition $(\mathcal{P}_*\mathcal{E},\theta)=\bigoplus_i(\mathcal{P}_*\mathcal{E}_i,\theta_i)\otimes \mathbb{C}^{m(i)}$ such that (i) $(\mathcal{P}_*\mathcal{E}_i,\theta_i)$ are $\mu_L$-stable, (ii) $\mu_L(\mathcal{P}_*\mathcal{E})=\mu_L(\mathcal{P}_*\mathcal{E}_i)$, (iii) $(\mathcal{P}_*\mathcal{E}_i,\theta_i)\centernot\simeq(\mathcal{P}_*\mathcal{E}_j,\theta_j)$ $(i\neq j)$. We call the decomposition $(\mathcal{P}_*\mathcal{E},\theta)=\bigoplus_i(\mathcal{P}_*\mathcal{E}_i,\theta_i)\otimes \mathbb{C}^{m(i)}$ the canonical decomposition.
\end{corollary}
\subsection{Filtered bundles}
\subsubsection{Local case}\label{local}
Let $U$ be an open neighborhood of $0\in\mathbb{C}^n$. Let $H_{U_i}:=U\cap\{z_i=0\}$ and $H_U:=\bigcup^l_{i=1}H_{U_i}$. Let $\mathcal{V}$ be a locally free $\mathcal{O}_U(\ast H)$-module. A filtered bundle $\mathcal{P}_*\mathcal{V}$ is a family of locally free $\mathcal{O}_U$-modules $\mathcal{P}_{\textit{\textbf{a}}}\mathcal{V}$ indexed by $\textit{\textbf{a}}\in\mathbb{R}^l$ such that 
\begin{itemize}
\item $\mathcal{P}_{\textit{\textbf{a}}}\mathcal{V}\subset\mathcal{P}_{\textit{\textbf{b}}}\mathcal{V}$ for $\textit{\textbf{a}}\leq\textit{\textbf{b}}$.
\item There exists a frame $(v_1,\dots,v_r)$ of $\mathcal{V}$ and tuples $a(v_j)\in\mathbb{R}$ $( j=1,\dots,l )$ such that 
\begin{equation*}
\mathcal{P}_{\textit{\textbf{b}}}\mathcal{V}=\bigoplus^r_{j=1}\mathcal{O}_U\bigg(\sum^l_{i=1}[b_i-a(v_j)]H_{U_i}\bigg)v_j.
\end{equation*}
Here for $c\in\mathbb{R}$, $[c]:=\mathrm{max}\{a\leq c| a\in\mathbb{Z}\}$.
\end{itemize}
Hence locally, a filtered bundle is a filtered sheaf that is locally free and has a frame compatible with filtration.
\subsubsection{Pullback of filtered bundles}
We use the same notation as in the previous section. Let $\varphi:\mathbb{C}^n\to\mathbb{C}^n$ be a map given by $\varphi(\xi_1,\dots,\xi_n)=(\xi^{m_1}_1,\dots,\xi^{m_{l}}_l,\xi_{l+1},\dots,\xi_n)$. We set $U':=\varphi^{-1}(U)$ and $H_{U',i}:=\varphi^{-1}(H_{U,i})$. We denote the induced ramified covering $U'\to U$ as $\varphi.$\par
For any $\textbf{\textit{b}}\in \mathbb{R}^l$, we set $\varphi^*(\textbf{\textit{b}})=(m_ib_i)\in\mathbb{R}^l$. Let $\mathcal{P}_*\mathcal{V}$ be a filtered bundle on $(U,H_U)$.

\subsubsection{Global case}
In this section, we assume $X$ to be a complex manifold and $H=\bigcup_{i\in\Lambda}H_i$ to be a normal crossing divisor of it.\par
Let $\mathcal{V}$ be a locally free $\mathcal{O}_X(\ast H)$-module. A filtered bundle $\mathcal{P}_*\mathcal{V}$ over $\mathcal{V}$ is a filtered sheaf over $\mathcal{V}$ such that it is locally written as in Section \ref{local}. We give some examples of filtered bundles.\par
Let $\mathcal{P}_*\mathcal{V}_1$ and $\mathcal{P}_*\mathcal{V}_2$ be filtered bundles. For $P\in H$, we take an admiisble coordinate neighborhood $(U_P,z_1,\dots,z_n)$ such that each and any $\mathcal{P}_{\textbf{\textit{a}}}\mathcal{V}_i|_{U_P}$ only depends on $\kappa_P(\textbf{\textit{a}})$. We define filtered bundles $\mathcal{P}_*(\mathcal{V}_1|_{U_P}\oplus\mathcal{V}_2|_{U_P})$, $\mathcal{P}_*\mathcal{V}_1|_{U_P}\otimes\mathcal{P}_*\mathcal{V}_2|_{U_P}$ and $\mathcal{P}_*(\mathcal{H}om(\mathcal{V}_1|_{U_P},\mathcal{V}_2|_{U_P}))$ on $U_P$ as,
\begin{align*}
\mathcal{P}_{\textbf{\textit{a}}}(\mathcal{V}_1|_{U_P}\oplus\mathcal{V}_2|_{U_P}):&=\mathcal{P}_{\textbf{\textit{a}}}\mathcal{V}_1|_{U_P}\oplus\mathcal{P}_{\textbf{\textit{a}}}\mathcal{V}_2|_{U_P},\\
\mathcal{P}_{\textbf{\textit{a}}}(\mathcal{V}_1|_{U_P}\otimes\mathcal{V}_2|_{U_P}):&=\sum_{\textbf{\textit{c}}_1+\textbf{\textit{c}}_2\leq \textbf{\textit{a}}}\mathcal{P}_{c_1}\mathcal{V}_1|_{U_P}\otimes\mathcal{P}_{c_2}\mathcal{V}_2|_{U_P},\\
\mathcal{P}_{\textbf{\textit{a}}}(\mathcal{H}om(\mathcal{V}_1|_{U_P},\mathcal{V}_2|_{U_P})):=\bigg\{&f\in\mathcal{H}om(\mathcal{V}_1|_{U_P},\mathcal{V}_2|_{U_P}) \mid f(\mathcal{P}_{\textbf{\textit{b}}}\mathcal{V}_1|_{U_P})\subset f(\mathcal{P}_{\textbf{\textit{a}+\textbf{\textit{b}}}}\mathcal{V}_2|_{U_P})(\forall \textbf{\textit{b}}\in\mathbb{R}^{l(P)})\bigg\}.
\end{align*}
Here $\textbf{\textit{a}}\in\mathbb{R}^{l(P)}$. We construct filtered bundles as above around for each $P\in H$. After taking a suitable covering of $X$, we can glue the filtered bundles and obtain unique filtered bundles $\mathcal{P}_*(\mathcal{V}_1\oplus\mathcal{V}_2)$, $\mathcal{P}_*(\mathcal{V}_1\otimes\mathcal{V}_2)$ and $\mathcal{P}_*(\mathcal{H}om(\mathcal{V}_1,\mathcal{V}_2))$ such that $\mathcal{P}_*(\mathcal{V}_1\oplus\mathcal{V}_2)|_{U_P}=\mathcal{P}_*(\mathcal{V}_1|_{U_P}\oplus\mathcal{V}_2|_{U_P})$, $\mathcal{P}_*(\mathcal{V}_1\otimes\mathcal{V}_2)|_{U_P}=\mathcal{P}_*(\mathcal{V}_1|_{U_P}\otimes\mathcal{V}_2|_{U_P})$ and $\mathcal{P}_*(\mathcal{H}om(\mathcal{V}_1,\mathcal{V}_2)|_{U_P})=\mathcal{P}_*(\mathcal{H}om(\mathcal{V}_1|_{U_P},\mathcal{V}_2|_{U_P}))$ holds for any $P\in H$. We denote these filtered bundles $\mathcal{P}_*\mathcal{V}_1\oplus\mathcal{P}_*\mathcal{V}_2$, $\mathcal{P}_*\mathcal{V}_1\otimes\mathcal{P}_*\mathcal{V}_2$ and $\mathcal{H}om(\mathcal{P}_*\mathcal{V}_1,\mathcal{P}_*\mathcal{V}_2)$.\par
\par
Let $\mathcal{P}_*\mathcal{V}$ be a filtered bundle and let $\mathcal{V}'\subset\mathcal{V}$ be a locally free sub $\mathcal{O}_X(\ast H)$-module of $\mathrm{rank}\mathcal{V}'<\mathrm{rank}\mathcal{V}.$ We obtain a filtered bundle $\mathcal{P}_*\mathcal{V}'$ (See section \ref{f sh}). 
\begin{remark}
Let $\mathcal{V}',\mathcal{V}''\subset\mathcal{V}$ be locally free subsheaves. We note that even if $\mathcal{V}=\mathcal{V}'\oplus\mathcal{V}''$ holds, $\mathcal{P}_*\mathcal{V}=\mathcal{P}_*\mathcal{V}'\oplus\mathcal{P}_*\mathcal{V}''$ does not always hold. Here $\mathcal{P}_*\mathcal{V}', \mathcal{P}_*\mathcal{V}''$ is the induced filtration from $\mathcal{P}_*\mathcal{V}$. We say that the $\mathcal{P}_*\mathcal{V}$ is compatible with decomposition if $\mathcal{P}_*\mathcal{V}=\mathcal{P}_*\mathcal{V}'\oplus\mathcal{P}_*\mathcal{V}''$holds.\par
We give a very easy example of a filtered bundle that is not compatible with decomposition. Let $U$ be an open neighborhood of $0\in\mathbb{C}$. Let $\mathcal{V}:=\mathcal{O}_U(\ast 0)e_1\oplus\mathcal{O}_U(\ast 0)e_2$. For every $a\in\mathbb{R}$, we set 
\begin{equation*}
\mathcal{P}_a\mathcal{V}:=\mathcal{O}_U\big([a] 0\big)e_1\oplus\mathcal{O}_U\bigg(\bigg[a+\frac{1}{2}\bigg] 0\bigg)e_2.
\end{equation*}
We set $\mathcal{V}_1:=\mathcal{O}_U(\ast 0)(e_1+e_2)$ and $\mathcal{V}_2:=\mathcal{O}_U(\ast 0)(e_1-e_2)$. It is easy to see that $\mathcal{V}=\mathcal{V}_1\oplus\mathcal{V}_2$ holds. Let $\mathcal{P}_*\mathcal{V}_1$ and $\mathcal{P}_*\mathcal{V}_2$ be the induced filtered bundle. The decomposition $\mathcal{V}=\mathcal{V}_1\oplus\mathcal{V}_2$ is not compatible with filtration. For example, take $a=\frac{1}{2}$. Then 
\begin{align*}
\mathcal{P}_{\frac{1}{2}}&\mathcal{V}=\mathcal{O}_U e_1\oplus\mathcal{O}_U\frac{e_2}{z},\\
\mathcal{P}_{\frac{1}{2}}\mathcal{V}_1&=\mathcal{O}_U(\ast 0)(e_1+e_2)\cap\mathcal{P}_{\frac{1}{2}}\mathcal{V}=\mathcal{O}_U(e_1+e_2),\\
\mathcal{P}_{\frac{1}{2}}\mathcal{V}_2&=\mathcal{O}_U(\ast 0)(e_1-e_2)\cap\mathcal{P}_{\frac{1}{2}}\mathcal{V}=\mathcal{O}_U(e_1-e_2).
\end{align*}
Hence the decomposition is not compatible with the filtration. Obviously, if we set $\mathcal{V}'_1:=\mathcal{O}_U(\ast 0)(e_1)$ and $\mathcal{V}'_2:=\mathcal{O}_U(\ast 0)(e_2)$, then the decomposition is compatible with the filtration.
\end{remark}
Let rank$\mathcal{V}=r$ and $\mathcal{P}_*\mathcal{V}$ be a filtered bundle over it. We obtain a filtered bundle $\mathcal{P}_*\mathcal{V}^{\otimes r}$ over $\mathcal{V}^{\otimes r}$ as above. We obtain a filtered bundle $\mathrm{det} (\mathcal{P}_*\mathcal{V})$ over $\mathrm{det}\mathcal{V}\subset\mathcal{V}^{\otimes r} $ by the canonical way. \par
We construct a filtered bundle over $\mathcal{P}^{(0)}_*(\mathcal{O}_X(\ast H))$ over $\mathcal{O}_X(\ast H)$. Let $P\in H$ and $(U_P,z_1,\dots,z_n)$ be the admissible coordinate of $P$. For $\textbf{\textit{a}}\in\mathbb{R}^{\Lambda}$, we define 
\begin{equation*}
\mathcal{P}^{(0)}_{\textbf{\textit{a}}}(\mathcal{O}_X(\ast H))|_{U_P}:=\mathcal{O}_X\bigg(\sum^{l(P)}_{i=1} [\kappa(\textbf{\textit{a}})_i]H_i\bigg)
\end{equation*}
here $\kappa(\textbf{\textit{a}})_i$ is the $i$-th component of $\kappa(\textbf{\textit{a}})$ and for $a\in\mathbb{R}$, $[a]:=\mathrm{max}\{n\in\mathbb{Z} | n\leq a\}$. We then glue the filtered bundle above and obtain the filtered bundle $\mathcal{P}^{(0)}_{\ast}(\mathcal{O}_X(\ast H))$. Let $\mathcal{P}_*\mathcal{V}$ be a filtered bundle over $\mathcal{V}$. We have a  filtered bundle $\mathcal{P}_*\mathcal{V}^\vee:=\mathcal{H}om(\mathcal{P}_*\mathcal{V},\mathcal{P}^{(0)}_*(\mathcal{O}_X(\ast H))$.
\subsubsection{Induced bundles and filtrations}
We use the same notation as the previous section.\par
Let $I\subset\Lambda$ be any subset and $\bm{\delta}_I\in\mathbb{R}^{\Lambda}$ be the element such that the $j$-th component is 0 if $j\in\Lambda\backslash I$ and 1 if $j\in\Lambda$. Let $H_I:=\bigcap_{i\in I}H_i$ and $\partial H_I:=H_I\backslash\big( \bigcup_{i\in\Lambda}H_i\big)$.  \par
Let $\mathcal{P}_*\mathcal{V}$ be filtered bundle over $(X,H)$. In this section, we introduce some subsheaves of $\mathcal{P}_{\textbf{\textit{a}}}\mathcal{V}|_{H_I}(\textbf{\textit{a}}\in\mathbb{R}^{\Lambda})$. We use these subsheaves to define  Chern characters for $\mathcal{P}_*\mathcal{V}$ in the next section. \par
Let $i\in\Lambda$. Let $\textbf{\textit{a}}\in\mathbb{R}^{\Lambda}$ and for $a_i-1< b\leq a_i$, let $\textbf{\textit{a}}(b,i):=\textbf{\textit{a}}+(b-a_i)\bm{\delta}_i$. We want to introduce a filtration on $\mathcal{P}_{\textbf{\textit{a}}}\mathcal{V}|_{H_i}$. First, we define $^iF_b(\mathcal{P}_{\textbf{\textit{a}}}\mathcal{V}|_{H_i})$ as 
\begin{equation*}
^iF_b(\mathcal{P}_{\textbf{\textit{a}}}\mathcal{V}|_{H_i}):=\mathcal{P}_{\textbf{\textit{a}}(b,i)}\mathcal{V}|_{H_i}\bigg/\mathcal{P}_{\textbf{\textit{a}}(a_i-1,i)}\mathcal{V}|_{H_i}.
\end{equation*}
This is a locally free $\mathcal{O}_{H_i}$-module and it is a subbundle of $\mathcal{P}_{\textbf{\textit{a}}}\mathcal{V}|_{H_i}.$ Hence $^iF_\ast$ gives a increasing filtration on $\mathcal{P}_{\textbf{\textit{a}}}\mathcal{V}|_{H_i}$ indexed by $(a_i-1,a_i]$.\par
For general $I\subset \Lambda$, we introduce a family of subbundle of $\mathcal{P}_{\textbf{\textit{a}}}\mathcal{V}|_{H_I}$. Let $\textbf{\textit{a}}_I$ be the image of $\textbf{\textit{a}}$ of the natural projection $\mathbb{R}^\Lambda\to\mathbb{R}^I$. Let $(\textbf{\textit{a}}_I-\bm{\delta}_I,\textbf{\textit{a}}_I]:=\prod_{i\in I}(a_i-1,a_i]$. For any $\textbf{\textit{b}}\in(\textbf{\textit{a}}_I-\bm{\delta}_I,\textbf{\textit{a}}_I]$, we set 
\begin{equation*}
^IF_{\textbf{\textit{b}}}(\mathcal{P}_{\textbf{\textit{a}}}\mathcal{V}|_{H_I}):=\bigcap_{i\in I}{^i}F_{b_i}\big(\mathcal{P}_{\textbf{\textit{a}}}\mathcal{V}|_{H_i}\big).
\end{equation*}  
From the local description of  filtered bundles, for any $P\in H_I$, there exists a neighborhood $X_P$ of P in $X$ and a non-canonical decomposition 
\begin{equation*}
\mathcal{P}_{\textbf{\textit{a}}}\mathcal{V}|_{X_P\cap H_I}=\bigoplus_{\textbf{\textit{b}}\in(\textbf{\textit{a}}_I-\bm{\delta}_I,\textbf{\textit{a}}_I] }\mathcal{G}_{P, \textbf{\textit{b}} }
\end{equation*}
such that the following holds for any $\textbf{\textit{c}}\in (\textbf{\textit{a}}_I-\bm{\delta}_I,\textbf{\textit{a}}_I]$
\begin{equation*}
^IF_{\textbf{\textit{c}}}(\mathcal{P}_{\textbf{\textit{a}}}\mathcal{V}|_{X_P\cap H_I})=\bigoplus_{\textbf{\textit{b}}\leq\textbf{\textit{c}}}\mathcal{G}_{P, \textbf{\textit{b}}}.
\end{equation*}
Hence for any $\textbf{\textit{c}}\in(\textbf{\textit{a}}_I-\bm{\delta}_I,\textbf{\textit{a}}_I]$, we obtain the following locally free $\mathcal{O}_{H_I}$-modules:
\begin{equation*}
^I\mathrm{Gr}^F_{\textbf{\textit{c}}}(\mathcal{P}_{\textbf{\textit{a}}}\mathcal{V}):=\frac{^IF_{\textbf{\textit{c}}}(\mathcal{P}_{\textbf{\textit{a}}}\mathcal{V}|_{H_I})}{\sum_{\textbf{\textit{b}}\lneq\textbf{\textit{c}}} {^I}F_{\textbf{\textit{b}}}(\mathcal{P}_{\textbf{\textit{a}}}\mathcal{V}|_{H_I})}.
\end{equation*}
Here $(b_i)=\textbf{\textit{b}}\lneq\textbf{\textit{c}}=(c_i)$ means that $b_i\leq c_i$ for any $i$ and $\textbf{\textit{b}}\neq\textbf{\textit{c}}$.
We note that $^I\mathrm{Gr}^F_{\textbf{\textit{c}}}(\mathcal{P}_{\textbf{\textit{a}}}\mathcal{V})$ forms a subbundle of $\mathcal{P}_{\textbf{\textit{a}}}\mathcal{V}|_{H_I}$ on the irreducible component of $H_I$.
\subsubsection{First Chern class and Second Chern class for filtered bundles}
We use the same notation as in the previous section.\par
In this section, we recall the definition of the first Chern class and the second Chern character for filtered bundles. Let $\mathcal{P}_*\mathcal{V}$ be a filtered bundle over $(X, H)$.
In Section \ref{slope}, we recalled the definition of the first Chren class for filtered sheaves. Since filtered bundles are filtered sheaves, the first Chern class of filtered bundles is defined as follows.
\begin{equation*}
c_1(\mathcal{P}_*\mathcal{V})=c_1(\mathcal{P}_{\textbf{\textit{a}}}\mathcal{V})-\sum_{i\in\lambda}\sum_{a_i-1<b\le a_i}b\cdot \mathrm{rank}^i\mathrm{Gr}^F_{b}(\mathcal{P}_{\textbf{\textit{a}}}\mathcal{V}|_{H_i})\cdot[H_i]\in H^2(X,\mathbb{R}).
\end{equation*}
Let $\textrm{Irr}(H_i\cap H_j)$ be the set of irreducible components of $H_i\cap H_j$. For $C\in \textrm{Irr}(H_i\cap H_j)$, let $[C]\in H^4(X,\mathbb{R})$ be the induced cohomology class and let $^C\mathrm{Gr}^F_{(c_i,c_j)}(\mathcal{P}_{\textbf{\textit{a}}}\mathcal{V})$ be the restriction of $^{(i,j)}\mathrm{Gr}^F_{(c_i,c_j)}(\mathcal{P}_{\textbf{\textit{a}}}\mathcal{V})$ to $C$. Let  $\iota_{i^\ast}:H^2(H_i,\mathbb{R})\to H^4(X,\mathbb{R})$ be the Gysin map induced by $\iota_i:H_i\to X$. The second Chern character for filtered bundles is defined as follows.
\begin{align*}
\mathrm{ch_2}(\mathcal{P}_*\mathcal{V}):=&\mathrm{ch_2}(\mathcal{P}_{\textbf{\textit{a}}}\mathcal{V})-\sum_{i\in\Lambda}\sum_{a_i-1<b\le a_i}b\cdot \iota_{i^\ast}(c_1(^i\mathrm{Gr}^F_{b}(\mathcal{P}_{\textbf{\textit{a}}}\mathcal{V}|_{H_i})))\\
&+\frac{1}{2}\sum_{i\in\Lambda}\sum_{a_i-1<b\le a_i}b^2\cdot\mathrm{rank}(^i\mathrm{Gr}^F_{b}(\mathcal{P}_{\textbf{\textit{a}}}\mathcal{V}))[H_i]^2\\
&+\frac{1}{2}\sum_{i,j\in\Lambda^2,i\neq j}\sum_{C\in \textrm{Irr}(H_i\cap H_j)}\sum_{a_i-1<c_i\leq a_i, a_j-1<c_j\leq a_j}c_i\cdot c_j\mathrm{rank} ^C\mathrm{Gr}^F_{(c_i,c_j)}(\mathcal{P}_{\textbf{\textit{a}}}\mathcal{V})\cdot [C].
\end{align*}
\subsection{Prolongation of vector bundles}
Let $X$ be a complex manifold and $H=\cup_{i\in\Lambda}H_i$ be a normal crossing hypersurface. Let $(E,\overline\partial_E)$ be a holomorphic vector bundle over $X \backslash H$ and $h$ be a hermitian metric of $E$. We define a presheaf $\widetilde{\mathcal{P}^h_{\textit{\textbf{a}}}E}$ on $X$  such that for an open set $U$ of $X$, $\widetilde{\mathcal{P}^h_{\textit{\textbf{a}}}E}(U)$ is a set of holomorphic section of $E$ on $U$ which satisfies the following growing condition along $U\cap H$:
\begin{itemize}
\item Let $P\in H$ and $(U_P,z_1,\dots,z_n)$ be an admissible neighborhood of $P$ such that $\overline{U}_P\subset U$. Let $\textbf{\textit{c}}:=\kappa_P(\textbf{\textit{a}})$. A holomorphic section $s$ of $E$ on $U$ is $s\in \mathcal{P}^h_{\textit{\textbf{a}}}E(U)$ when $s$ satisfies the following estimate on $U_P$
\begin{equation*}
|s|_h \leq O\bigg(\prod^{c}_{i=1}|z_i|^{-c_i-\epsilon}\bigg)
\end{equation*}
for any $\epsilon\in\mathbb{R}_{>0}$.
\end{itemize}
We denote the sheafification of $\widetilde{\mathcal{P}^h_{\textit{\textbf{a}}}E}$ as $\mathcal{P}^h_{\textit{\textbf{a}}}E$.
We obtain a $\mathcal{O}_X$-module $\mathcal{P}^h_{\textit{\textbf{a}}}E$ and we obtain a $\mathcal{O}_X(\ast H)$-module $\mathcal{P}^h_{\ast}E:=\bigcup_{\textit{\textbf{a}}\in\mathbb{R}^{\Lambda}}\mathcal{P}^h_{\textit{\textbf{a}}}E$.
\begin{definition}
Let $\mathcal{P}_{\ast}\mathcal{V}$ be a filtered bundle over $(X, H)$. Let $(E,\overline\partial_E)$ be a holomorphic bundle obtained from the restriction of  $\mathcal{V}$ to $X- H$. Let $h$ be a hermitian metric of $E$. $h$ is called adapted if $\mathcal{P}^h_{\ast}E=\mathcal{P}_{\ast}\mathcal{V}$ stands.
\end{definition}
We remark that in general, we do not know whether $\mathcal{P}^h_{\ast}E$ is locally free or not. However, it was proved in \cite[Theorem 21.3.1]{M2} that when the metric $h$ is $\textit{acceptable}$ and $\mathrm{det}(E,\overline\partial_E,h)$ is flat, $\mathcal{P}^h_{\ast}E$ is locally free. We say that $h$ is acceptable when the following condition holds:
\begin{itemize}
\item Let $P\in H$ and let $(U_P,z_1,\dots,z_n)$ be an admissible neighborhood of $P$. We regard as $U_P=\prod^n_{i=1}\{|z_i|<1\}$. Let $g_P$ be a Poincar\'e like metric on $U_P\backslash U_P\cap H$. The metric $h$ is called acceptable around $P$ when the curvature of the Chern connection is bounded with respect to $g_P$ and $h$. $h$ is called acceptable if it is acceptable around any $P\in H$.
\end{itemize}
\subsection{Good filtered Higgs bundle}
Throughout this section, we assume $X$ to be a complex manifold and $H=\bigcup_{i\in\Lambda}H_i$ to be a simple normal crossing hypersurface of it.
\subsubsection{Good set of Irregular values}
Let $P\in H$. Let $(U_P,z_1,\dots,z_n)$ be an admissible coordinate around $P$. We denote the stalk of $\mathcal{O}_X(\ast H)$ at $P$ as $\mathcal{O}_X(\ast H)_P$. Let $f\in \mathcal{O}_X(\ast H)_P$. If $\mathcal{O}_{X,P}$, we set $\mathrm{ord}(f)=(0,\dots0)\in\mathbb{R}^{l(P)}$. If there exsits a $g\in\mathcal{O}_{X,P}$, $g(P)\neq0$ and a  $\textbf{\textit{n}}\in\mathbb{Z}_{< 0}^{l(P)}$ such that $g=f\prod z_i^{-n_i}$, we set $\mathrm{ord}(f)=\textbf{\textit{n}}$. Otherwise, $\mathrm{ord}(f)$ is not defined. Note that when $\mathrm{dim}X=1$ and when $f$ has at least  a simple pole at $P$, then $\mathrm{ord}(f)$ is the usual order.\par
For any $\mathfrak{a}\in\mathcal{O}_X(\ast H)_P/\mathcal{O}_{X,P}$, we take a lift $\tilde{\mathfrak{a}}\in\mathcal{O}_X(\ast H)_P$. If $\mathrm{ord}(\mathfrak{a})$ is defined, we set $\mathrm{ord}(\mathfrak{a}):=\mathrm{ord}(\tilde{\mathfrak{a}})$. Otherwise $\mathrm{ord}(\mathfrak{a})$ is not defined. $\mathrm{ord}(\mathfrak{a})$ does not depend on the lift.\par
Let $\mathcal{I}_P\subset\mathcal{O}_X(\ast H)_P/\mathcal{O}_{X,P}$ be finite subset. We say that $\mathcal{I}_P$ is called a good set of irregular values if 
\begin{itemize}
\item $\mathrm{ord}(\mathfrak{a})$ is defined for any $\mathrm{ord}(\mathfrak{a})\in\mathcal{I}_P$
\item $\mathrm{ord}(\mathfrak{a-b})$ is defined for any $\mathrm{ord}(\mathfrak{a}), \mathrm{ord}(\mathfrak{b})\in\mathcal{I}_P$
\item $\{\mathrm{ord}(\mathfrak{a-b})| \mathfrak{a}, \mathfrak{b}\in\mathcal{I}_P\}$ is totally orded with respect to the order $\leq_{\mathbb{Z}^{l(P)}}$.
\end{itemize}
Note that when $\mathrm{dim}X=1$, then any finite subset of $\mathcal{O}_X(\ast H)_P/\mathcal{O}_{X,P}$ is a good set of irregular values.
\subsubsection{Good filtered Higgs bundle}
Let $(\mathcal{P}_{*}\mathcal{V},\theta)$ be a filtered Higgs bundle. Let $P\in X$ and let $\mathcal{O}_{X,\widehat{P}}$ be the completion of the local ring $\mathcal{O}_{X,P}$ with respect to its maximal ideal.\par
We say that $(\mathcal{P}_{*}\mathcal{V},\theta)$ is called unramifiedly good at $P$ if there exisits a good set of irregular values $\mathcal{I}_P$ and exsits a decomposition of Higgs bunlde
\begin{equation*}
(\mathcal{P}_{*}\mathcal{V},\theta)\otimes\mathcal{O}_{X,\widehat{P}}=\bigoplus_{\mathfrak{a}\in\mathcal{I}_P}(\mathcal{P}_{*}\mathcal{V}_{\mathfrak{a}},\theta_{\mathfrak{a}})
\end{equation*}
such that $(\theta_{\mathfrak{a}}-\mathrm{d}\widetilde{\mathfrak{a}}\mathrm{Id}_{\mathcal{V}_{\mathfrak{a}}})\mathcal{P}_{\textit{\textbf{a}}}\mathcal{V}_{\mathfrak{a}}\subset\mathcal{P}_{\textit{\textbf{a}}}\mathcal{V}_{\mathfrak{a}}\otimes\Omega^1_X(\mathrm{log}H)$ for evey $\textit{\textbf{a}}\in\mathbb{R}^\Lambda$. Here $\widetilde{\mathfrak{a}}$ is the lift of $\mathfrak{a}$.\par
$(\mathcal{P}_{*}\mathcal{V},\theta)$ is called good at $P$ if there exists a neighborhood $U_P$ and  a covering map $\varphi_P:U'_P\to U_P$ ramified over $H\cap U_P$ such that $\varphi^\ast_P(\mathcal{P}_{*}\mathcal{V},\theta)$ is unramified good at $\varphi^{-1}_P(P)$.\par
$(\mathcal{P}_{*}\mathcal{V},\theta)$ is called good (resp. unramifiedly good) if it is good (resp. unramfiedly good) at any point of $H$.
\subsection{Good Wild Harmonic Bundles}
\subsubsection{Local condition for Higgs fields}
Let $U:=\prod^n_{i=1}\big\{|z_i|<1\big\}$ and $H_{U_i}:=U\cap\{z_i=0\}$ and $H_U:=\bigcup^l_{i=1}H_{U_i}$. Let $(E,\overline\partial_E,\theta)$ be a Higgs bundle on $U - H_U$. The Higgs field $\theta$ has an expression
\begin{equation*}
\theta=\sum^l_{i=1}\frac{F_i}{z_i}dz_i+\sum^n_{i=l+1}G_idz_i.
\end{equation*}
Let $T$ be a formal variable. We have characteristic polynomials 
 \begin{equation*}
 \mathrm{det}(T-F_i(z))=\sum_k A_{i,k}(z) T^k, \mathrm{det}(T-G_i(z))=\sum_k B_{i,k}(z)T^k
 \end{equation*}
 where $ A_{i,k}(z), B_{i,k}(z) $ are holomorphic functions on $U-H_U$.
 \begin{definition}
 We say that $\theta$ is tame if $A_{i,k}(z), B_{i,k}(z)$ are holomorphic functions on $U$ and if the restriction of $A_{i,k}$ to $H_{U_i}$ are constant for any $j$ and $k$.
 \end{definition}
 \begin{definition}
 \hfill
  \begin{itemize}
 \item We say that $\theta$ is unramfiedly good if there exists a good set of irregular value $\mathrm{Irr}(\theta)\subset M(U, H_U)\big/H(X)$ and a decomposition
 \begin{equation*}
 (E,\theta)=\bigoplus_{\mathfrak{a}\in\mathrm{Irr}(\theta)}(E_{\mathfrak{a}},\theta_{\mathfrak{a}})
 \end{equation*}
 such that each $\theta_{\mathfrak{a}}-d\widetilde{\mathfrak{a}}\cdot\mathrm{Id}_{E_{\mathfrak{a}}}$ is tame. Here $\widetilde{\mathfrak{a}}$ is the lift of $\mathfrak{a}$.
 \item For  $e\in\mathbb{Z}_{>0}$, we define the covering map $\phi_e:U\to U$ as $\phi(z_1,\dots,z_n)=(z^e_1,\dots,z^e_l,z_{l+1},\dots,z_n)$. 
  We say that $\theta$ is good if there exists a $e\in\mathbb{Z}_{>0}$ and the pullback of $(E,\overline\partial_E,\theta)$ by $\phi_e$ is unramifiedly good.
 \end{itemize}
 \end{definition}
 \subsubsection{Global condition of Higgs fields and Good Wild Harmonic bundles}
 Let $X$ be a complex manifold and $H$ be a normal crossing hypersurface. Let $(E,\overline\partial_E, \theta)$ be a Higgs bundle on $X-H$.
 \begin{definition}
 \hfill
  \begin{itemize}
\item We say that $\theta$ is (unramifiedly) good at $P\in H$ if it is (unramifiedly) good on an admissible coordinate neighborhood of $P$.
\item We say that $\theta$ is (unramifiedly) good on $(X,H)$ if it is (unramifiedly) good for any $P\in H$.
 \end{itemize}
 \end{definition}
We next recall \textit{good wild harmonic bundles}. Let $h$ be a pluri-harmonic metric of $(E,\overline\partial_E, \theta)$ (i.e.
$(E,\overline\partial_E, \theta, h)$ is a harmonic bundle on $X-H$).
 \begin{definition}
 We say that  $(E,\overline\partial_E, \theta, h)$ is a (unramifiedly) good wild harmonic bundle on $(X,H)$ if $\theta$ is (unramifiedly) good on $(X,H)$.
   \end{definition} 
 \subsection{Kobayashi-Hitchin Correspondence}
Let $X$ be a connected smooth projective variety and $H$ be a simple normal crossing divisor. Let $L$ be any ample line bundle.
\par
 In \cite{M1, M2} Mochizuki proved that there is a one-on-one correspondence between $\mu_L$-polystable good filtered Higgs bundles with vanishing Chern classes and good wild harmonic bundles. This correspondence is called Kobayashi-Hitchin Correspondence.
 \begin{proposition}[{\cite[Proposition 13.6.1 and 13.6.4]{M1}}]\label{mor}
 Let $(E,\theta,h)$ be a good wild harmonic bundle on $(X,H)$.
 \begin{itemize}
 \item $(\mathcal{P}^h_{\ast}E,\theta)$ is  $\mu_L$-polystable with $\mu_L(\mathcal{P}^h_{\ast}E)=0$.
 \item $c_1(\mathcal{P}^h_{\ast}E)=0$ and $ \int_X\mathrm{ch}_2(\mathcal{P}_{\ast}\mathcal{V})c_1(L)^{\mathrm{dim}X-2}=0$ holds.
 \item Let $h'$ be another pluri-harmonic metric of $(E,\theta, h)$ such that $\mathcal{P}^{h'}_{\ast}E=\mathcal{P}^h_{\ast}E.$ Then there exists a decomposition of the Higgs bundle $(E,\theta)=\oplus_i(E_i,\theta_i)$ such that (i) the decomposition is orthogonal with respect to both $h$ and $h'$, (ii) $h|_{E_i}=a_ih'|_{E_i}$ for some $a_i>0$.
  \end{itemize}
 \end{proposition}
 \begin{theorem}[{\cite[Theorem 2.23.]{M2}}]\label{KH0}
 Let $(\mathcal{P}_{\ast}\mathcal{V},\theta)$ be a good filtered Higgs bundle on $(X, H)$ and $(E,\overline\partial_E,\theta)$ be the Higgs bundle on $X\backslash H$ which is the restriction of $(\mathcal{P}_{\ast}\mathcal{V},\theta)$. \par
 Suppose that $(\mathcal{P}_{\ast}\mathcal{V},\theta)$ is $\mu_L$-polystable and satisfies the following vanishing condition:
 \begin{equation}\label{vanish}
 \mu_L(\mathcal{P}_{\ast}\mathcal{V})=0, \int_X\mathrm{ch}_2(\mathcal{P}_{\ast}\mathcal{V})c_1(L)^{\mathrm{dim}X-2}=0.
  \end{equation}
  Then there exists a pluri-harmonic metric $h$ for $(E,\overline\partial_E,\theta)$ such that $(\mathcal{V},\theta)|_{X\backslash H}\simeq(E,\theta)$ extends to $(\mathcal{P}_{\ast}\mathcal{V},\theta)\simeq(\mathcal{P}^h_{\ast}E,\theta)$.
  \end{theorem}
\begin{remark}
We note that Theorem \ref{KH0} was proved not only for the Higgs bundles but for all $\lambda$-flat bundles. The $\lambda=1$ case was established in \cite{M1}.
\end{remark}
  \section{Good filtered Higgs bundles with skew-symmetric pairings}\label{main s}
\subsection{Pairings of filtered bundle}
Throughout this section, we assume $X$ to be a smooth projective variety and let $H=\bigcup_{i\in \Lambda} H_i$ be a normal crossing divisor of it, and $L$ to be an ample line bundle on $X$. However, we only use this assumption in Section 4.1.4. The results in other sections can generalized for any complex manifold and normal crossing hypersurfaces.
\subsubsection{Pairings of locally free $\mathcal{O}_X(\ast H)$-modules}
Let $\mathcal{O}_X(\ast H)$ be the sheaf of meromorphic function on $X$ whose poles are contained in $H$. We recall the pairings of $\mathcal{O}_X(\ast H)$-modules  following \cite{qm1}. \par
Let $\mathcal{V}$ be a locally free $\mathcal{O}_X(\ast H)$-module of finite rank. Let $\mathcal{V}^{\vee}:=\mathcal{H}om_{\mathcal{O}_X(\ast H)}(\mathcal{V},\mathcal{O}_X(\ast H))$ be the dual of $\mathcal{V}.$
The determinant bundle of $\mathcal{V}$ is denoted by 
det$(\mathcal{V}):=\bigwedge^{\textrm{rank}\mathcal{V}}\mathcal{V}.$ There exists a natural isomorphism $\textrm{det}(\mathcal{V}^\vee)\simeq\textrm{det}(\mathcal{V})^\vee.$ For a morphism $f:\mathcal{V}_1\to\mathcal{V}_2$ of locally free $\mathcal{O}_X(\ast H)$-modules, we have the dual $f^\vee:\mathcal{V}_2^\vee\to\mathcal{V}_1^\vee$. If rank($\mathcal{V}_1)$=rank$(\mathcal{V}_2$), then we have the induced morphism $\mathrm{det}(f):\mathrm{det}(\mathcal{V}_1)\to\mathrm{det}(\mathcal{V}_2)$.\par
A pairing $P$ of a pair of locally free $\mathcal{O}_X(\ast H)$-modules $\mathcal{V}_1$ and $\mathcal{V}_2$ is a morphism $P:\mathcal{V}_1\otimes\mathcal{V}_2\to\mathcal{O}_X(\ast D)$. It induces a morphism $\Psi_P:\mathcal{V}_1\to\mathcal{V}_2^\vee$ by $\Psi_P(u)(v):=P(u,v).$ Let $\textrm{ex}:\mathcal{V}_1\otimes\mathcal{V}_2\simeq\mathcal{V}_2\otimes\mathcal{V}_1$ be the morphism defined by $\textrm{ex}(u\otimes v)=v\otimes u$. We obtain a pairing $P\circ\textrm{ex}:\mathcal{V}_2\otimes\mathcal{V}_1\to\mathcal{O}_X(\ast H)$. We have $\Psi_P^\vee=\Psi_{P\circ\textrm{ex}}$. If rank $\mathcal{V}_1$=rank $\mathcal{V}_2$, we obtain the induced pairing $\textrm{det}P:\textrm{det}(\mathcal{V}_1)\otimes\textrm{det}(\mathcal{V}_2)\to\mathcal{O}_X(\ast H)$. We have $\textrm{det}(\Psi_P)=\Psi_{\textrm{det}(P)}.$\par
A pairing $P$ is called non-degenerate if $\Psi_P$ is an isomorphism. It is equivalent to that $P\circ\textrm{ex}$ is non-degenerate. It is also equivalent to be $\textrm{det}P$ is non-degenerate. If $P$ is non-degenerate, we obtain a pairing $P^\vee$ of $\mathcal{V}_2^\vee$ and $\mathcal{V}_1^\vee$ defined by $P\circ(\Psi_P^{-1}\otimes\Psi_{P\circ\textrm{ex}}).$\par
A pairing $P$ of locally free $\mathcal{O}_X(\ast H)$-module $\mathcal{V}$ is a morphism $P:\mathcal{V}\otimes\mathcal{V}\to\mathcal{O}_X(\ast H)$. It is called skew-symmetric if $P\circ\textrm{ex}=-P.$ Note that $\textrm{det}(P)$ is natural defined in this case. If $P$ is non-degenerate, then rank$:\mathcal{V}$ must be even and we have induced pairing $P^\vee$ of $\mathcal{V}^\vee.$
\subsubsection{Pairings of filtered bundles}
Let $\mathcal{P}_*\mathcal{V}_i$ $(i=1,2)$ be a filtered bundle on $(X, H)$. A pairing $P$ of $\mathcal{P}_*\mathcal{V}_1$ and $\mathcal{P}_*\mathcal{V}_2$ is a morphism between filtered bunlde 
\begin{equation*}
P: \mathcal{P}_*\mathcal{V}_1\otimes \mathcal{P}_*\mathcal{V}_2\to \mathcal{P}^{(0)}_*(\mathcal{O}_X(\ast H)).
\end{equation*}
We obtain a pairing $P\circ\textrm{ex}$ of $\mathcal{P}_*\mathcal{V}_2$ and $\mathcal{P}_*\mathcal{V}_1$.

From the pairing $P$, we also obtain the  following morphism 
\begin{equation*}
\Psi_P:\mathcal{P}_*\mathcal{V}_1\to \mathcal{P}_*\mathcal{V}_2^\vee.
\end{equation*}
\begin{definition}
P is called perfect if the morphism $\Psi_P$ is an isomorphism of filtered bundles.
 \end{definition}

Let $\mathcal{V}'_i\subset \mathcal{V}_i$ be a locally free $\mathcal{O}_X(\ast H)$-submodules. We also assume $\mathcal{V}'_i$ are saturated i.e. $\mathcal{V}_i/\mathcal{V}'_i$ are locally free. From a pairing $P$ of $\mathcal{P}_*\mathcal{V}_1$ and $\mathcal{P}_*\mathcal{V}_2$, we have the induced pairing $P'$ for $\mathcal{P}_*\mathcal{V}'_1$ and $\mathcal{P}_*\mathcal{V}'_2$. We have a sequence of sheaves:
\begin{equation*}
\mathcal{V}'_1\overset{i_1}{\longrightarrow}\mathcal{V}_1\overset{\Psi_P}{\longrightarrow}\mathcal{V}_2^\vee \overset{i^\vee_2}{\longrightarrow}\mathcal{V}_2^{'\vee}
\end{equation*}
where $i_1$ is the canonical inclusion and $i^\vee_2$ is the dual of the canonical inclusion. Note that $\Psi_{P'}=i_2^\vee\circ\Psi_P\circ i_1.$ Let $\mathcal{U}_1:=\mathrm{ker}({i_2^\vee\circ\Psi_P})$. It is a subsheaf of $\mathcal{V}_1$.
 \begin{lemma}\label{decomp}
 If $P$ and $P'$ are perfect, then we have the decomposition $\mathcal{V}_1=\mathcal{V}'_1\oplus \mathcal{U}_1.$
 \end{lemma}
 \begin{proof}
 We have the following short exact sequence of sheaves:
 \begin{equation*}
 0\longrightarrow\mathcal{V}'_1\longrightarrow \mathcal{V}_1 \longrightarrow \mathcal{V}_1/\mathcal{V}'_1 \longrightarrow 0.
 \end{equation*}

Since $P$ and $P'$ are non-degenerate, we have another short exact sequence of sheaves:
 \begin{equation*}
 0\longrightarrow \mathcal{V}'_1\longrightarrow \mathcal{V}_1 \longrightarrow \mathcal{U}_1 \longrightarrow 0.
 \end{equation*}
By the standard argument of sheaves,  we have $\mathcal{U}_1\simeq  \mathcal{V}_1/\mathcal{V}'_1.$ 
Hence we have $\mathcal{V}_1=\mathcal{V}'_1\oplus \mathcal{U}_1.$
 \end{proof}
\subsubsection{Skew-symmetric pairings of filtered bundles}
Let $\omega$ be a skew-symmetric pairing of a filtered bundle $\mathcal{P}_*\mathcal{V}$ on $(X, H)$. 
Let $\mathcal{V}'\subset \mathcal{V}$ be a saturated locally free $\mathcal{O}_X(\ast H)$-submodule. Let $(\mathcal{V}')^{\perp\omega}$ be the kernel of the following composition:
\begin{equation*}
\mathcal{V}\overset{\Psi_{\omega}}\longrightarrow \mathcal{V}^\vee\overset{i^\vee}\longrightarrow \mathcal{V}^{'\vee}
\end{equation*}
where $i^\vee$ is the dual of the canonical inclusion. Let $\omega'$ be the induced skew-symmetric pairing of $\mathcal{P}_*\mathcal{V}'$. The next Lemma is the special case of Lemma \ref{decomp}.
\begin{lemma}
If $\omega$ and $\omega'$ are perfect, then we have the decomposition $\mathcal{V}=\mathcal{V}'\oplus(\mathcal{V}')^{\perp\omega}$.
\end{lemma}
\subsection{Skew-symmetric pairings of good filtered Higgs bundle}
Throughout this section, we assume $X$ to be a smooth projective variety and let $H=\bigcup_{i\in \Lambda} H_i$ be a normal crossing divisor of it, and $L$ to be an ample line bundle on $X$. 
\subsubsection{Skew-symmetric pairings of Higgs bundle}
 \begin{definition}
 A skew-symmetric pairing $\omega$ on a good filtered Higgs bundle $(\mathcal{P}_*\mathcal{V},\theta)$ over $(X,H)$ is a skew-symmetric pairing $\omega$ of $\mathcal{P}_*\mathcal{V}$ such that $\omega(\theta\otimes \mathrm{Id})=-\omega(\mathrm{Id}\otimes \theta)$. 
  \end{definition}
 When $(\mathcal{P}_*\mathcal{V},\theta)$ has a skew-symmetric pairing $\omega$, we have an induced morphism $\Psi_\omega:(\mathcal{P}_*\mathcal{V}, \theta)\to(\mathcal{P}_*\mathcal{V}^\vee,-\theta^\vee)$ between good filtered Higgs bundles. We also obtain a symmetric pairing det$(\omega)$ of $(\mathrm{det}(\mathcal{P}_*\mathcal{V}),\mathrm{tr} \theta$).
\subsubsection{Harmonic bundles with skew-symmetric structure}\label{KH par}
We use the same notation as the last section. Let $(E,\overline\partial_E,\theta,h)$ be a good wild harmonic bundle on  $(X, H)$. Let $\omega$ be a symplectic structure of the harmonic bundle $(E,\overline\partial_E,\theta,h)$. By Proposition \ref{mor}, we obtain a  $\mu_L$-polystable good filtered Higgs bundle $(\mathcal{P}_*^hE, \theta)$ with vanishing Chern classes.
 \begin{lemma}
$\omega$ induces a perfect skew-symmetric pairing for the Higgs bundle $(\mathcal{P}^h_*E,\theta)$.
 \end{lemma}
\begin{proof}
Since $\omega$ is compatible with $h$, it induces an isomorphism $\Psi_\omega:\mathcal{P}^h_*E\to\mathcal{P}^{h^{\vee}}_*E^\vee$. Since $\mathcal{P}^{h^{\vee}}_*E^\vee$ is naturally isomorphic to $(\mathcal{P}^{h}_*E)^\vee$, $\omega$ induces a perfect pairing for $\mathcal{P}^h_*E$.
\end{proof}
As a consequence, we have the following.
 \begin{proposition}\label{w-f}
 Let $(E,\overline\partial_E,\theta, h)$ be a good wild harmonic bundle equipped with symplectic structure $\omega$. Then $(\mathcal{P}_*^hE, \theta)$ is a  $\mu_L$-polystable
good filtered Higgs bundle equipped with a perfect skew-symmetric pairing $\omega$ and satisfies the vanishing condition (\ref{vanish}).
 \end{proposition}
\subsection{Kobayashi-Hitchin correspondence with skew-symmetry}\label{KH part}
Throughout this section, we assume $X$ to be a smooth projective variety and let $H=\bigcup_{i\in \Lambda} H_i$ be a normal crossing divisor of it, and $L$ to be an ample line bundle on $X$. 
 \subsubsection{Basic polystable object (1)}
 Let $(\mathcal{P}_*\mathcal{V}, \theta)$ be a stable good filtered Higgs bundle of degree 0 such that $(\mathcal{P}_*\mathcal{V}, \theta)\simeq(\mathcal{P}_*\mathcal{V}^\vee, -\theta^\vee)$. Let $P$ be a pairing of a filtered bundle 
 \begin{equation*}
 P:\mathcal{P}_*\mathcal{V}\otimes \mathcal{P}_*\mathcal{V}\to \mathcal{P}^{(0)}_*(\mathcal{O}_X(\ast H))
 \end{equation*}
 such that it induces an isomorphism $\Psi_{P}:(\mathcal{P}_*\mathcal{V}, \theta)\to(\mathcal{P}_*\mathcal{V}^\vee, -\theta^\vee)$. If there is another pairing $P'$ which induces an isomorphism $\Psi_{P'}$, then since a stable bundle is simple there exists an $\alpha\in\mathbb{C}$ such that $P'=\alpha P$. 
  \begin{lemma}
 Either one of $P\circ \mathrm{ex}=P$ or $P\circ \mathrm{ex}=-P$ holds.
 \end{lemma}
 \begin{proof} This was proved in \cite[Lemma 3.19]{qm1}. The claim follows from the fact that there exists a $\alpha\in\mathbb{C}$ such that $\Psi^\vee_P=\alpha\Psi_P$, $(\Psi^\vee_P)^\vee=\Psi_P$, $\Psi_{P\circ\mathrm{ex}}=\Psi^\vee_{P}$.
 \end{proof}
 Let $C_{\mathbb{C}^l}$ be a symmetric pairing of $\mathbb{C}^l$ defined by $C(\bm{x},\bm{y}):=\sum_ix_iy_i$ for $\bm{x}, \bm{y}\in\mathbb{C}^l$. Let $\omega_{\mathbb{C}^{2k}}$ be a skew-symmetric pairing of $\mathbb{C}^{2k}$ defined by $\omega_{\mathbb{C}^{2k}}(\bm{x},\bm{y}):$$=\sum_i(x_{2i-1}y_{2i}-x_{2i-1}y_{2i})$. If $P_1$ is a symmetric pairing then $P_1\otimes \omega_{\mathbb{C}^{2k}}$ is a skew-symmetric pairing for $(E,\theta)\otimes \mathbb{C}^{2k}.$  If $P_1$ is skew-symmetric then $P_1\otimes C_{{\mathbb{C}^l}}$ is a skew-symmetric pairing for $(\mathcal{P}_*\mathcal{V}, \theta)\otimes \mathbb{C}^l$.
 \begin{lemma}
 Suppose that $(\mathcal{P}_*\mathcal{V}, \theta)\otimes{\mathbb{C}^l}$ is equipped with a perfect skew-symmetric pairing $\omega$.
 \begin{itemize}\label{B}
  \item If $P_1$ is symmetric, then l is an even number 2k and there exists an automorphism $\tau$ for $\mathbb{C}^{2k}$ such that $(Id\otimes\tau)^*\omega=P_1\otimes\omega_{\mathbb{C}^{2k}}$.
  \item If $P_1$ is skew-symmetric then there exists an automorphism $\tau$ for $\mathbb{C}^{l}$ such that $(Id\otimes\tau)^*\omega=P_1\otimes C_{\mathbb{C}^l}$.
 \end{itemize}
 \end{lemma}
 \begin{proof}
We only give the outline of the proof for the case when $P$ is symmetric. The other case can be proved similarly.\par
   Let $\{e_i\}^l_{i=1}$ be the canonical base of $\mathbb{C}^{l}$. Since $\omega$ is a perfect skew-symmetric pairing of $(\mathcal{P}_*\mathcal{V}, \theta)\otimes{\mathbb{C}^{l}}$, it induces an isomorphism $\Psi_\omega:(\mathcal{P}_*\mathcal{V}, \theta)\otimes\mathbb{C}^{l}\to(\mathcal{P}_*\mathcal{V}^\vee, -\theta^\vee)\otimes\mathbb{C}^{l}$.  Let $\Psi_{\omega,ij}$ be the composition of 
\begin{equation*}
(\mathcal{P}_*\mathcal{V}, \theta)\otimes e_i \overset{i}{\longrightarrow}(\mathcal{P}_*\mathcal{V}, \theta)\otimes\mathbb{C}^{l}\overset{\Psi_{\omega}}{\longrightarrow}(\mathcal{P}_*\mathcal{V}^\vee, -\theta^\vee)\otimes\mathbb{C}^{l}\overset{pr_j}{\longrightarrow}(\mathcal{P}_*\mathcal{V}^\vee, -\theta^\vee)\otimes e_j
\end{equation*}
where $i$ is the inclusion and $pr_j$ is the projection. Either one $\Psi_{\omega,ij}=0$ or $\Psi_{\omega,ij}=\alpha_{ij}\Psi_{P_1}$ for a $\alpha_{ij}\in\mathbb{C}$ holds. Since $\omega$ is a perfect pairing, $(\alpha_{ij})_{i,j}$ is non-degenerate matrix and since $\omega$ is skew-symmetric and $P$ is symmetric,  $(\alpha_{ij})_{i,j}$ is a skew-symmetric matrix. Hence $l$ is an even number $2k$  and there is an automorphism $\tau$ which we want.
 \end{proof}
 \begin{lemma}
 There is an unique harmonic metric $h_0$ on $\mathcal{V}|_{X/D}$ such that (1) it is adapted to $\mathcal{P}_*\mathcal{V}$ and (2) $\Psi_{P_1}$ is isometric with respect to $h_0$ and $h_0^\vee.$ 
 \end{lemma}
 \begin{proof}
 By Theorem \ref{SBBM}, we have a harmonic metric $h$ on $\mathcal{V}|_{X/D}$ which is adapted to $\mathcal{P}_*\mathcal{V}$. Let $h^\vee$ be the induced harmonic metric of $\mathcal{V}^\vee|_{X\setminus D}$ by $h$, which is also adapted to $\mathcal{P}_*\mathcal{V}^\vee.$ Since $\Psi_{P}:(\mathcal{P}_*\mathcal{V}, \theta)\to(\mathcal{P}_*\mathcal{V}^\vee, -\theta^\vee)$ is an isomorphism, $\Psi_{P}^*(h^\vee)$ is also a harmonic metric which is adapted to $\mathcal{P}_*\mathcal{V}$. Since the adapted harmonic metric for a stable Higgs bundle is unique up to positive constant, we have an $a>0$ such that  $\Psi_{P}^*(h^\vee)=a^2h.$ Set $h_0:=ah$ then we obtain the desired metric. The uniqueness is clear.
  \end{proof}
     \begin{lemma}\label{B1}
 \hfill
  \begin{itemize}
\item  For any hermitian metric $h_{\mathbb{C}^l}$ of $\mathbb{C}^l$, $h_0\otimes h_{\mathbb{C}^l}$ is a harmonic metric  of $\mathcal{V}|_{X/D}\otimes\mathbb{C}^l$ which is adapted to $\mathcal{P}_*\mathcal{V}\otimes\mathbb{C}^l$. Conversely, for any harmonic metric $h$ on $\mathcal{V}|_{X/D}\otimes\mathbb{C}^l$ which is adapted to $\mathcal{P}_*\mathcal{V}\otimes\mathbb{C}^l$, there is a hermitian metric $h_{\mathbb{C}^l}$ of $\mathbb{C}^l$ such that $h=h_0\otimes h_{\mathbb{C}^l}$.
\item If $P_1$ is symmetric (resp. skew-symmetric), a harmonic metric $h_0\otimes h_{\mathbb{C}^l}$  of $\mathcal{V}|_{X/D}\otimes\mathbb{C}^l$ is compatible with $P_1\otimes \omega_{\mathbb{C}^l}$ (resp. $P_1\otimes C_{\mathbb{C}^{l}}$) if and only if  $h_{\mathbb{C}^l}$ is compatible with $\omega_{\mathbb{C}^l}$ (resp. $ C_{\mathbb{C}^{l}})$.
  \end{itemize}
  \end{lemma}
  \begin{proof}
  The first claim follows from the uniqueness of the harmonic metric to a parabolic structure. See \cite[Corollary 13.6.2]{M2}.\par
  The second claim follows from the following argument: Let $E_i(i=1,2)$ be complex vector bundles and $h_i(i=1,2)$ be hermitian metrics for $E_i$. Let $P_i(i=1,2)$ be pairings for $E_i$ (i.e. $P_i$ is a section of $E^\vee_i\otimes E^\vee_i$) and $\Psi_{P_i}:E_i\to E_i^\vee$ be the indued morphisms. Let $h_1\otimes h_2$ be the hermitian metric of $E_1\otimes E_2$ induced by $h_i$ and $P_1\otimes P_2$ be the pairing of $E_1\otimes E_2$ induced by $P_i$. Let $u_i\otimes v_i(i=1,2)$ be  sections of $E_1\otimes E_2$. $h_i\otimes h_2$ and $P_1\otimes P_2$ are defined as $h_i\otimes h_2(u_1\otimes v_1,u_2\otimes v_2)=h_1(u_1,u_2)h_2(v_1,v_2)$ and $P_1\otimes P_2(u_1\otimes v_1,u_2\otimes v_2)=P_1(u_1,u_2)P_2(v_1,v_2)$. Hence $\Psi_{P_1\otimes P_2}=\Psi_{P_1}\otimes\Psi_{P_2}$ and $(h_1\otimes h_2)^\vee(\Psi_{P_1\otimes P_2}(u_1\otimes v_1),\Psi_{P_1\otimes P_2}(u_2\otimes v_2))=h_1^\vee(\Psi_{P_1}(u_1),\Psi_{P_1}(u_2))h_2^\vee(\Psi_{P_2}(v_1),\Psi_{P_2}(v_2))$ holds. Once we apply this discussion to $h_0\otimes h_{\mathbb{C}^l}$ and $P_1\otimes\omega_{\mathbb{C}^l}$ or $P_1\otimes C_{\mathbb{C}^{l}}$, the second claim follows.
  \end{proof}
  \subsubsection{Basic polystable objects (2)}
  Let $(\mathcal{P}_*\mathcal{V}, \theta)$ be a stable good filtered Higgs bundle that satisfies the vanishing condition (\ref{vanish}) and $(\mathcal{P}_*\mathcal{V}, \theta)$ $\centernot{\simeq}$ $(\mathcal{P}_*\mathcal{V}^\vee, -\theta^\vee)$. We set $\mathcal{P}_*\widetilde{\mathcal{V}}:=\mathcal{P}_*\mathcal{V}\oplus \mathcal{P}_*\mathcal{V}^\vee$ and set $\widetilde{\theta}:=\theta\oplus -\theta^\vee$. Then we obtain a Higgs bundle $(\mathcal{P}_*\widetilde{\mathcal{V}},\widetilde{\theta})$. We have a  naturally defined perfect skew-symmetric pairing of $(\mathcal{P}_*\widetilde{\mathcal{V}},\widetilde{\theta}),$
  \begin{equation*}
  \widetilde{\omega}_{(\mathcal{P}_*\mathcal{V}, \theta)}:(\mathcal{P}_*\widetilde{\mathcal{V}},\widetilde{\theta})\otimes(\mathcal{P}_*\widetilde{\mathcal{V}},\widetilde{\theta})\to\mathcal{P}^{(0)}_*(\mathcal{O}_X(\ast H))
  \end{equation*}
  such that $\widetilde{\omega}_{(\mathcal{P}_*\mathcal{V}, \theta)}((u_1,v^\vee_1),(u_2,v^\vee_2))=v^\vee_1(u_2)-v^\vee_2(u_1)$ for any local section $(u_1,v^\vee_1),(u_2,v^\vee_2)$ of $\mathcal{P}_*\widetilde{\mathcal{V}}$. $(\mathcal{P}_*\widetilde{\mathcal{V}},\widetilde{\theta}, \widetilde{\omega}_{(\mathcal{P}_*\mathcal{V}, \theta)})$ forms a Higgs bundle with a perfect skew-symmetric pairing.
  \begin{lemma}\label{B2}
  Suppose $\big((\mathcal{P}_*\mathcal{V}, \theta)\otimes\mathbb{C}^{l_1}\big)\oplus\big((\mathcal{P}_*\mathcal{V}^\vee, -\theta^\vee)\otimes\mathbb{C}^{l_2}\big)$ is equipped with a perfect skew-symmetric pairing $\omega.$ Then we have $l_1=l_2$ and there exists an isomorphism $(\mathcal{P}_*\widetilde{\mathcal{V}},\widetilde{\theta})\otimes\mathbb{C}^{l_1}\simeq(\mathcal{P}_*\mathcal{V}, \theta)\otimes\mathbb{C}^{l_1}\oplus(\mathcal{P}_*\mathcal{V}^\vee, -\theta^\vee)\otimes\mathbb{C}^{l_2}$ such that under the isomorphism, $\widetilde{\omega}_{(\mathcal{P}_*\mathcal{V}, \theta)}\otimes C_{\mathbb{C}^{l_1}}=\omega$ holds.
    \end{lemma}
    \begin{proof}
    We have  one-dimensional subspaces $L_1\subset\mathbb{C}^{l_1}$ and $L_2\subset\mathbb{C}^{l_2}$ such that the restriction of $\omega$ to $\big((\mathcal{P}_*\mathcal{V}, \theta)\otimes L_1\big)\oplus\big((\mathcal{P}_*\mathcal{V}^\vee, -\theta^\vee)\otimes L_2\big)$ is not identically zero. We define $\Psi_{\omega,12}$ to be the composition of 
    \begin{align*}
   (\mathcal{P}_*\mathcal{V}, \theta)\otimes L_1 \overset{i}{\longrightarrow}\big((\mathcal{P}_*\mathcal{V}, \theta)&\otimes L_1\big)\oplus\big((\mathcal{P}_*\mathcal{V}^\vee, -\theta^\vee)\otimes L_2\big)
   \\ &\overset{\Psi_{\omega}}{\longrightarrow}\big((\mathcal{P}_*\mathcal{V}^\vee, -\theta^\vee)\otimes L_1^\vee\big)\oplus\big((\mathcal{P}_*\mathcal{V}, \theta)\otimes L_2^\vee\big)\overset{pr_2}{\longrightarrow}(\mathcal{P}_*\mathcal{V}, \theta)\otimes L_2^\vee
    \end{align*}
    where $i$ and $pr_2$ are the canonical inclusion and the canonical projection. We define $\Psi_{\omega,11}$,$\Psi_{\omega,21}$ and $\Psi_{\omega,22}$ in the same manner. Since $(\mathcal{P}_*\mathcal{V}, \theta)$ $\centernot{\simeq}$ $(\mathcal{P}_*\mathcal{V}^\vee, -\theta^\vee)$, we obtain $\Psi_{\omega,11}=0, \Psi_{\omega,22}=0$ and $\Psi_{\omega,12}=\alpha \textrm{Id}_{(\mathcal{P}_*\mathcal{V}, \theta)}, \Psi_{\omega,21}=\beta \textrm{Id}_{(\mathcal{P}_*\mathcal{V}, \theta)}$ for some $\alpha,\beta\in\mathbb{C}$. Since $\omega$ is a skew-symmetric pairing, we have $\beta=-\alpha$. Hence $\omega=\alpha \widetilde{\omega}_{(\mathcal{P}_*\mathcal{V}, \theta)}$. In particular the restriction of $\omega$ to $\big((\mathcal{P}_*\mathcal{V}, \theta)\otimes L_1\big)\oplus\big((\mathcal{P}_*\mathcal{V}^\vee, -\theta^\vee)\otimes L_2\big)$ induces a perfect skew-symmetric pairing on it. Hence we obtain an orthonormal decomposition with respect to $\omega$:
    \begin{equation*}
(\mathcal{P}_*\mathcal{V}\otimes\mathbb{C}^{l_1})\oplus\big(\mathcal{P}_*\mathcal{V}^\vee\otimes\mathbb{C}^{l_2}\big)\simeq\big(\mathcal{P}_*\mathcal{V}\otimes L_1\big)\oplus\big(\mathcal{P}_*\mathcal{V}^\vee\otimes L_2\big)\oplus\mathcal{P}_*\mathcal{V}'.
   \end{equation*}
   It is preserved by the Higgs field and the induced Higgs field to $\mathcal{P}_*\mathcal{V}'$ is isomorphic to $\big((\mathcal{P}_*\mathcal{V}, \theta)\otimes\mathbb{C}^{l_1-1}\big)\oplus\big((\mathcal{P}_*\mathcal{V}^\vee, -\theta^\vee)\otimes\mathbb{C}^{l_2-1}\big)$. We obtain the claim by induction.
                \end{proof}
By using $C_{\mathbb{C}^{l}}$, we can identify $\mathbb{C}^{l}$ and it's dual $(\mathbb{C}^{l})^\vee$. Then we can induce a perfect skew-symmetric pairing  $\widetilde{\omega}_{(\mathcal{P}_*\mathcal{V}, \theta)}\otimes C_{\mathbb{C}^{l}}$ on 
\begin{equation*}
(\mathcal{P}_*\widetilde{\mathcal{V}}, \widetilde{\theta})\otimes\mathbb{C}^{l}=\big((\mathcal{P}_*\mathcal{V}, \theta)\otimes\mathbb{C}^{l}\big)\oplus\big((\mathcal{P}_*\mathcal{V}^\vee, -\theta^\vee)\otimes(\mathbb{C}^{l})^\vee\big)
\end{equation*}          
by the canonical way.\par
 We obtain the induced harmonic metric $h^\vee_0$ on  $\mathcal{V}^\vee|_{X/D}$ which is adapted to $\mathcal{P}_*\mathcal{V}^\vee.$ 
\begin{lemma}\label{B3}
\hfill
\begin{itemize}
\item Let $h_{\mathbb{C}^l}$ be any hermitian metric on $\mathbb{C}^l$. Let $h_{\mathbb{C}^l}^\vee$ denote the induced hermitian metric on $(\mathbb{C}^{l})^\vee.$ Then, $(h_0\otimes h_{\mathbb{C}^l})\oplus (h_0^\vee\otimes h_{\mathbb{C}^l}^\vee )$ is a harmonic metric of $(\mathcal{P}_*\widetilde{\mathcal{V}}, \widetilde{\theta})\otimes\mathbb{C}^{l}$ such that it is compatible with $\widetilde{\omega}_{(\mathcal{P}_*\mathcal{V}, \theta)}\otimes C_{\mathbb{C}^{l}}.$
\item Conversely, let h be any harmonic metric of $(\mathcal{P}_*\widetilde{\mathcal{V}}, \widetilde{\theta})\otimes\mathbb{C}^{l}$ which is compatible with  $\widetilde{\omega}_{(\mathcal{P}_*\mathcal{V}, \theta)}\otimes C_{\mathbb{C}^{l}}$. Then there exists a hermitian metric $h_{\mathbb{C}^l}$ of   $\mathbb{C}^l$ such that $h=(h_0\otimes h_{\mathbb{C}^l})\oplus (h_0^\vee\otimes h_{\mathbb{C}^l}^\vee)$.
\end{itemize}   
\end{lemma}    
\begin{proof}
The compatibility of $(h_0\otimes h_{\mathbb{C}^l})\oplus (h_0^\vee\otimes h_{\mathbb{C}^l}^\vee )$ with $\widetilde{\omega}_{(\mathcal{P}_*\mathcal{V}, \theta)}\otimes C_{\mathbb{C}^{l}}$ follows from the argument in the second claim of Lemma \ref{B1}. The second claim follows from \cite[Lemma 3.25]{qm1}.
\end{proof}
\subsubsection{Polystable objects} \label{KH par2} 
Let $(\mathcal{P}_*\mathcal{V}, \theta)$ be a polystable good filtered Higgs bundle of degree 0 on $X$ equipped with a perfect skew-symmetric pairing $\omega$.  Let
\begin{equation*}
(\mathcal{P}_*\mathcal{V}, \theta)=\sum_i(\mathcal{P}_*\mathcal{V}_i, \theta_i)\otimes \mathbb{C}^{n(i)}
\end{equation*}
be the canonical decomposition. Since the perfect skew-symmetric pairing $\omega$ induces an isomorophism $(\mathcal{P}_*\mathcal{V}, \theta)\simeq(\mathcal{P}_*\mathcal{V}^\vee, -\theta^\vee),$ each $(\mathcal{P}_*\mathcal{V}_i, \theta_i)\otimes \mathbb{C}^{n(i)}$ is a basic polystable object we observed above. Hence the next proposition is deduced from previous sections.
\begin{proposition}\label{B4}
There exist stable Higgs bundles $(\mathcal{P}_*\mathcal{V}_i^{(0)}, \theta_i^{(0)})$ $(i=1,\dots,p(0))$, $(\mathcal{P}_*\mathcal{V}_i^{(1)}, \theta_i^{(1)})$ $(i=1,\dots,p(1))$ and $(\mathcal{P}_*\mathcal{V}_i^{(2)}, \theta_i^{(2)})$ $(i=1,\dots,p(2))$ of degree 0 on $X$ such that the following holds.
\begin{itemize}
\item $(\mathcal{P}_*\mathcal{V}_i^{(0)}, \theta_i^{(0)})$ is equipped with a symmetric pairing $P_i^{(0)}$.
\item  $(\mathcal{P}_*\mathcal{V}_i^{(1)}, \theta_i^{(1)})$ is equipped with a skew-symmetric pairing $P_i^{(1)}$.
\item  $(\mathcal{P}_*\mathcal{V}_i^{(2)}, \theta_i^{(2)})$ $\centernot{\simeq}$ $(\mathcal{P}_*\mathcal{V}_i^{(2)}, -\theta_i^{(2)})^\vee$.
\item There exists positive integers $l(a,i)$ and an isomorphism 
\begin{align*}
(\mathcal{P}_*\mathcal{V}, \theta)\simeq\bigoplus_{i=1}^{p(0)}(\mathcal{P}_*\mathcal{V}_i^{(0)}, \theta_i^{(0)})\otimes\mathbb{C}^{2l(0,i)}\oplus&\bigoplus_{i=1}^{p(1)}(\mathcal{P}_*\mathcal{V}_i^{(1)}, \theta_i^{(1)})\otimes\mathbb{C}^{l(1,i)}\\
&\oplus \bigoplus_{i=1}^{p(2)}\bigg(\big((\mathcal{P}_*\mathcal{V}_i^{(2)}, \theta_i^{(2)})\otimes\mathbb{C}^{l(2,i)}\big)\oplus\big((\mathcal{P}_*\mathcal{V}_i^{(2)}, -\theta_i^{(2)})^\vee\otimes(\mathbb{C}^{l(2,i)})^\vee\big)\bigg).
\end{align*}
Under this isomorphism, $\omega$ is identified with the direct sum of $P_{i}^{(0)}\otimes \omega_{\mathbb{C}^{2l(0,i)}}$, $P_{i}^{(1)}\otimes C_{\mathbb{C}^{l(1,i)}}$ and $\widetilde{\omega}_{(E_i^{(2)}, \theta_i^{(2)})}\otimes C_{\mathbb{C}^{l(2,i)}}$
\item $(\mathcal{P}_*\mathcal{V}_i^{(a)},\theta_i^{(a)})$ $\centernot{\simeq}$ $(\mathcal{P}_*\mathcal{V}_j^{(a)}, \theta_j^{(a)})$ $(i\neq j)$ for a=0,1,2, and $(\mathcal{P}_*\mathcal{V}_i^{(2)}, \theta_i^{(2)})$ $\centernot{\simeq}$ $(\mathcal{P}_*\mathcal{V}_j^{(2)}, -\theta_j^{(2)})^\vee$ for any $i, j.$
\end{itemize}
\end{proposition}
\begin{proof}
It follows from Lemma \ref{B} and Lemma \ref{B2}.
\end{proof} 
Let $h_i^{(a)}$ $(a=0,1)$ be the unique harmonic metrics of $(\mathcal{V}_i^{(a)},\theta_i^{(a)})|_{X\setminus D}$ such that $\mathrm{(i)}$ $h_i^{(a)}$ is adapted to $(\mathcal{P}_*\mathcal{V}_i^{(a)},\theta_i^{(a)})$, $\textrm{(ii)}$ $\Psi_{P_i^{(a)}}$ is isomoteric with respect to $h_i^{(a)}$ and $(h_i^{(a)})^\vee$. Let $h_{i}^{(2)}$ be any harmonic metrics of $(\mathcal{V}_i^{(2)}, \theta_i^{(2)})|_{X\setminus D}$ which is adapted to $(\mathcal{P}_*\mathcal{V}_i^{(2)}, \theta_i^{(2)})$.
\begin{proposition}\label{B5}
There exists a harmonic metric h of  $(\mathcal{V}, \theta)|_{X\setminus D}$ such that $\mathrm{(i)}$ h is adapted to $\mathcal{P}_*\mathcal{V}$,  $\mathrm{(ii)}$ it is compatible with $\omega.$ Moreover, we have the following.
 \begin{itemize}
 \item Let $h_{\mathbb{C}^{2l(0,i)}}$ be a hermitian metric of $\mathbb{C}^{2l(0,i)}$ compatible with $\omega_{\mathbb{C}^{2l(0,i)}}$. Let $h_{\mathbb{C}^{l(1,i)}}$ be a hermitian metric of $\mathbb{C}^{l(1,i)}$ compatible with $C_{\mathbb{C}^{l(1,i)}}$. Let $h_{\mathbb{C}^{l(2,i)}}$ be any hermitian metric on $\mathbb{C}^{l(2,i)}$. Then,
 \begin{equation}\label{shape cpt}
 \bigoplus_{i=1}^{p(0)}h_i^{(0)}\otimes h_{\mathbb{C}^{2l(0,i)}}\oplus\bigoplus_{i=1}^{p(1)}h_i^{(1)}\otimes h_{\mathbb{C}^{l(1,i)}}\oplus \bigoplus_{i=1}^{p(2)}\bigg(\big((h_{i}^{(2)}\otimes h_{\mathbb{C}^{l(2,i)}})\oplus\big( (h_{i}^{(2)})^\vee\otimes (h_{\mathbb{C}^{l(2,i)}})^\vee \big)\bigg)
   \end{equation} 
   is a harmonic metric which satisfies the condition $\textrm{(i)}, \textrm{(ii)}$.
   \item Conversely, if h is a harmonic metric of $(\mathcal{V}, \theta)|_{X\setminus D}$ which satisfies the condition $\textrm{(i)}$ and $\textrm{(ii)}$, then it has the form of (\ref{shape cpt}).
   \end{itemize}
\end{proposition}
\begin{proof}
The first claim follows from Proposition \ref{B4}. The second claim follows from Lemma \ref{B1} and Lemma \ref{B3}.
\end{proof}
\subsubsection{An equivalence}
In this section, we state the Kobayashi-Hitchin correspondence with skew symmetry.
Let $(E,\overline\partial_E,\theta,h)$ be a good wild harmonic bundle with symplectic structure $\omega$. From section \ref{KH par}, we obtain a good filtered Higgs bundle $(\mathcal{P}_*^hE, \theta)$ satisfying the vanishing condition (\ref{vanish}) equipped with a perfect skew-symmetric pairing. From section \ref{KH par2}, we also have the converse. As a result, we have the following.
\begin{theorem}\label{KH}
Let $X$ be a smooth projective variety and $H$ be a normal crossing divisor of X.\par
 The following objects are equivalent on $(X,H)$
    \begin{itemize}
  \item  Good wild harmonic bundles with a symplectic structure.
  \item  Good filtered polystable Higgs bundles with a perfect skew-symmetric pairing satisfying the vanishing condition (\ref{vanish}).
      \end{itemize}
  \end{theorem}
\begin{proof}
In section \ref{KH par}, we proved that from a good wild harmonic bundle with a symplectic structure we obtain a good filtered Higgs bundle satisfying the vanishing condition (\ref{vanish}) equipped with a perfect skew-symmetric pairing. We have the opposite side from section \ref{KH par2}. 
\end{proof}
The compact case is straightforward from Theorem \ref{KH}. However, for the compact case, we do not have to assume $X$ to be projective. In particular, the statement holds for arbitrary K\"ahler manifolds.
\begin{corollary}\label{KH cpt}
Let $X$ be a compact K\"ahler manifold. The following objects are equivalent on $X$.
\begin{itemize}
\item Harmonic bundles with a symplectic structure.
\item Polystable Higgs bundles with vanishing Chern classes with a perfect skew-symmetric pairing.
\end{itemize}
\end{corollary}


\begin{thebibliography}{99}
\bibitem{BB} O. Biquard, P. Boalch Wild non-abelian Hodge theory on curves, Compos. Math. \textbf{140} (2004), 179-204.
\bibitem{Co} K. Corlette, Flat $G$-bundles with canonical metrics, J. Differential Geom.\textbf{28} (1988), no. 3, 361-382.
\bibitem{Do} S. K. Donaldson, Twisted harmonic maps and the self-duality equations, Proc. London Math. Soc. (3) \textbf{55}(1987), no. 1, 127-131.
\bibitem{GGM} Oscar Garc\'ia-Prada, Peter B Gothen, and Ignasi Mundet i Riera, Higgs bundles and surface group representations in the real symplectic group, Journal of Topology \textbf{6} (2013), no. 1, 64-118.
\bibitem{H}  N. J. Hitchin, The self-duality equations on a Riemann surface, Proc. London Math. Soc. (3) \textbf{55} (1987), 59-126.
\bibitem{Ko}  S. Kobayashi, Differential geometry of complex vector bundles, Publications of the MSJ, 15. Kan\^o Memorial Lectures, 5. Princeton University Press, Princeton, NJ;
Iwanami Shoten, Tokyo, 1987.
\bibitem{qm1} Q. Li, T. Mochizuki, Harmonic metrics of generically regular semisimple Higgs bundles on non-compact Riemann surfaces, Tunisian Journal of Mathematics (2023), no. \textbf{4}, 663-711.
\bibitem{M0} T. Mochizuki, Kobayashi-Hitchin correspondence for tame harmonic bundles and an application, Ast\'erisque (2006), no. \textbf{309},  viii+117.
\bibitem{M1}  T. Mochizuki, Wild harmonic bundles and wild pure twistor $D$-modules, Ast\'erisque (2011), no. \textbf{340}, x+607
\bibitem{M2}  T. Mochizuki, Good wild harmonic bundles and good filtered Higgs bundles, SIGMA Symmetry Integrability
Geom. Methods Appl. \textbf{17} (2021), Paper No. 068, 66 pp
\bibitem{S1}C. Simpson, Constructing variations of Hodge structure using Yang-Mills theory and application to uniformization, J. Amer. Math. Soc. \textbf{1} (1988), 867-918.
\bibitem{S2} C. Simpson, Harmonic bundles on non-compact curves, J. Amer. Math. Soc. \textbf{3} (1990), 713-770.
\bibitem{S4} C.Simpson, Higgs bundles and local systems. Inst. Hautes Etudes Sci. Publ. Math. No. \textbf{75} (1992), 5-95.
\bibitem{S3} C. Simpson,  The Hodge filtration on non-abelian cohomology. Algebraic geometry-Santa Cruz 1995 (1997): 217-281.
\end{thebibliography}
\end{document}